\documentclass[letterpaper,10pt]{article}

\usepackage{amssymb}
\usepackage{amsmath}
\usepackage{amsthm}
\usepackage[labelsep=period,skip=6pt,labelfont=bf]{caption}
\usepackage{enumerate}
\usepackage[letterpaper,nohead,nomarginpar,centering,margin=1in]{geometry}
\usepackage{graphicx}
\usepackage[hyperfootnotes=false]{hyperref}
\usepackage[numbers,square,sort&compress]{natbib}
\usepackage[tiny]{titlesec}
\usepackage{doi}
\usepackage{tikz-cd}

\hypersetup{colorlinks,breaklinks,linkcolor=blue,urlcolor=blue,anchorcolor=blue,citecolor=blue}

\DeclareMathOperator{\spn}{span}
\DeclareMathOperator{\ran}{ran}

\theoremstyle{plain}
\newtheorem{thm}{Theorem}
\newtheorem{cor}[thm]{Corollary}
\newtheorem{lem}[thm]{Lemma}
\newtheorem{prop}[thm]{Proposition}

\theoremstyle{definition}
\newtheorem{defn}[thm]{Definition}

\begin{document}

\begin{center}
  \textbf{SPATIOTEMPORAL PATTERN EXTRACTION \\BY SPECTRAL ANALYSIS OF VECTOR-VALUED OBSERVABLES} \\
  \medskip
  Dimitrios Giannakis\footnote{Corresponding author. Email address: \url{dimitris@cims.nyu.edu}.}\\
  \emph{Courant Institute of Mathematical Sciences, New York University}\\
  \smallskip
  Abbas Ourmazd, Joanna Slawinska \\
  \emph{Department of Physics, University of Wisconsin-Milwaukee}\\
  \smallskip
  Zhizhen Zhao\\
  \emph{Department of Electrical and Computer Engineering, University of Illinois at Urbana-Champaign}
\end{center}

\bigskip

\begin{center}
    \textbf{Abstract}
\end{center}

We present a data-driven framework for extracting complex spatiotemporal patterns generated by ergodic dynamical systems. Our approach, called Vector-valued Spectral Analysis (VSA), is based on an eigendecomposition of a kernel integral operator acting on a Hilbert space of vector-valued observables of the system, taking values in a space of functions (scalar fields) on a spatial domain. This operator is constructed by combining aspects of the theory of operator-valued kernels for machine learning with delay-coordinate maps of dynamical systems.  In contrast to conventional eigendecomposition techniques, which decompose the input data into pairs of temporal and spatial modes with a separable, tensor product structure, the patterns recovered by VSA can be manifestly non-separable, requiring only a modest number of modes to represent signals with intermittency in both space and time. Moreover, the kernel construction naturally quotients out dynamical symmetries in the data, and exhibits an asymptotic commutativity property with the Koopman evolution operator of the system, enabling decomposition of multiscale signals into dynamically intrinsic patterns. Application of VSA to the Kuramoto-Sivashinsky model demonstrates significant performance gains in efficient and meaningful decomposition over eigendecomposition techniques utilizing scalar-valued kernels.   

\bigskip

\section{Introduction}

Spatiotemporal pattern formation is ubiquitous in physical, biological, and engineered systems, ranging from molecular-scale reaction-diffusion systems, to engineering- and geophysical-scale convective flows, and astrophysical flows, among many examples  \cite{CrossHohenberg93,AhlersEtAl09,FungEtAl16}. The mathematical models for such systems are generally formulated by means of partial differential equations (PDEs), or coupled ordinary differential equations, with dissipation playing an important role in the development of low-dimensional effective dynamics on attracting subsets of the state space \cite{ConstantinEtAl89}. In light of this property, many pattern forming systems are amenable to analysis by empirical, data-driven techniques, complementing the scientific understanding gained from first-principles approaches. 

Historically, many of the classical Proper Orthogonal Decomposition (POD) and Principal Component Analysis (PCA) techniques for spatiotemporal pattern extraction have been based on the spectral properties of temporal and spatial covariance operators estimated from snapshot data \citep[][]{AubryEtAl91,HolmesEtAl96}. In Singular Spectrum Analysis (SSA) and related algorithms \cite{BroomheadKing86,VautardGhil89,GhilEtAl02}, combining this approach  with delay-coordinate maps of dynamical systems \cite{PackardEtAl80,Takens81,SauerEtAl91,Robinson05,DeyleSugihara11} generally improves the representation of the information content of the data in terms of a few meaningful modes. More recently, advances in machine learning and applied harmonic analysis \cite{ScholkopfEtAl98,BelkinNiyogi03,CoifmanEtAl05,CoifmanLafon06,Singer06,VonLuxburgEtAl08,BerrySauer16b} have led to techniques for recovering temporal and spatial patterns through the eigenfunctions of kernel integral operators  (e.g., heat operators) defined intrinsically in terms of a Riemannian geometric structure of the data. In particular, in a family of techniques called Nonlinear Laplacian Spectral Analysis (NLSA) \cite{GiannakisMajda12a}, and independently in \cite{BerryEtAl13}, the diffusion maps algorithm \cite{CoifmanLafon06} was combined with delay-coordinate maps to extract spatiotemporal patterns through the eigenfunctions of a kernel integral operator adept at capturing distinct and physically meaningful timescales in individual eigenmodes from multiscale high-dimensional signals.         

At the same time, spatial and temporal patterns have been extracted from eigenfunctions of Koopman \cite{MezicBanaszuk04,Mezic05,RowleyEtAl09,GiannakisEtAl15,WilliamsEtAl15,BruntonEtAl17,DasGiannakis17,Giannakis17} and Perron-Frobenius \cite{DellnitzJunge99} operators governing the evolution of observables and probability measures, respectively, in dynamical systems \cite{BudisicEtAl12,EisnerEtAl15}. Koopman eigenfunction analysis is also related to the Dynamic Mode Decomposition (DMD) algorithm \cite{Schmid10} and Linear Inverse Model techniques \cite{Penland89}. An advantage of these approaches is that they target operators defined intrinsically for the dynamical system generating the data, and thus able, in principle, to recover temporal and spatial patterns of higher physical interpretability and utility in predictive modeling than POD and kernel integral operator based approaches. In practice, however, the Koopman and Perron-Frobenius operators tend to have significantly more complicated spectral properties (e.g., non-isolated eigenvalues and/or continuous spectra), hindering the stability and convergence of data-driven approximation techniques. In \cite{GiannakisEtAl15,Giannakis17,DasGiannakis17}, these issues were addressed through an approximation scheme for the generator of the Koopman group with rigorous convergence guarantees, utilizing a data-driven orthonormal basis of the Hilbert space of the dynamical system acquired through diffusion maps. There, it was also shown that the eigenfunctions recovered by kernel integral operators defined on delay-coordinate mapped data (e.g., the covariance and heat operators in SSA and NLSA, respectively) in fact converge to Koopman eigenfunctions in the limit of infinitely many delays, indicating a  deep connection between these two branches of data analysis algorithms.           

All of the techniques described above recover from the data a set of temporal patterns and a corresponding set of spatial patterns, sometimes referred to as ``chronos'' and ``topos'' modes, respectively \cite{AubryEtAl91}. In particular, for a dynamical system with a state space $ X $ developing patterns in a physical domain $ Y $, each chronos mode, $ \varphi_j $, corresponds to a scalar- (real- or complex-) valued function on $ X $, and the corresponding topos mode, $ \psi_j $, corresponds to a scalar-valued function on $ Y $. Spatiotemporal reconstructions of the data with these approaches thus correspond to linear combinations of tensor product patterns of the form $ \varphi_j \otimes \psi_j $, mapping pairs of points $ ( x, y ) $ in the product space $ \Omega = X \times Y $ to the number $ \varphi_j( x ) \psi_j( y ) $. For a dynamical system possessing a compact invariant set $ A \subseteq X $ (e.g., an attractor) with an ergodic invariant measure, the chronos modes effectively become scalar-valued functions on $ A $, which may be of significantly smaller dimension than $ X $, increasing the robustness of approximation of these modes from finite datasets. 

Evidently, for spatiotemporal signals $ F( x, y ) $ of high complexity, tensor product patterns, with separable dependence on $ x $ and $ y $, can be highly inefficient in capturing the properties of the input signal. That is, the number $l$ of such patterns needed to recover $ F $ at high accuracy via a linear superposition 
\begin{equation}
  \label{eqPODDecomp}
  F \approx  F_l = \sum_{j=0}^l \varphi_j \otimes \psi_j 
\end{equation}
is generally large, with none of the individual patterns $ \varphi_j \otimes \psi_j $  being representative of $ F $. In essence, the problem is similar to that of approximating a non-separable space-time signal in a tensor product basis of temporal and spatial basis functions. Another issue with tensor product decompositions based on scalar-valued eigenfunctions is that in the presence of nontrivial spatial symmetries, the recovered patterns are oftentimes pure symmetry modes (e.g., Fourier modes in a periodic domain with translation invariance), with minimal dynamical significance and physical interpretability \cite{AubryEtAl93,HolmesEtAl96}. 
 
Here, we present a framework for spatiotemporal pattern extraction, called Vector-valued Spectral Analysis (VSA), designed to alleviate the shortcomings mentioned above. The fundamental underpinning of VSA is that time-evolving spatial patterns have a natural structure as vector-valued observables on the system's state space, and thus data analytical techniques operating on such spaces are likely to offer maximal descriptive efficiency and physical insight.  We show that eigenfunctions of kernel integral operators on vector-valued observables, constructed by combining aspects of the theory of operator-valued kernels \cite{MicchelliPontil05,CaponnettoEtAl08,CarmeliEtAl10} with delay-coordinate maps of dynamical systems \cite{PackardEtAl80,Takens81,SauerEtAl91,Robinson05,DeyleSugihara11}: a) Are superior to conventional algorithms in capturing signals with intermittency in both space and time; b) Naturally incorporate any underlying dynamical symmetries, eliminating redundant modes and thus improving physical interpretability of the results; c) Have a correspondence with Koopman operators, allowing detection of intrinsic dynamical timescales; and, d) Can be stably approximated via data-driven techniques that provably converge in the asymptotic limit of large data. 

The plan of this paper is as follows. Section~\ref{secBackground} introduces the class of dynamical systems under study, and provides an overview of data analysis techniques based on scalar kernels. In Section~\ref{secVSA}, we present the VSA framework for spatiotemporal pattern extraction using operator-valued kernels, and in Section~\ref{secVSAProperties} discuss the behavior of the method in the presence of dynamical symmetries, as well as its correspondence with Koopman operators. Section~\ref{secDataDriven} describes the data-driven implementation of VSA. In Section~\ref{secApplications}, we present applications to the  Kuramoto-Sivashinsky (KS) PDE model \cite{KuramotoTsuzuki76,Sivashinsky77} in chaotic regimes.  Our primary conclusions are described in Section~\ref{secConclusions}.  Technical results and descriptions of basic properties of kernels and Koopman operators are collected in four appendices. An application of VSA to a toy spatiotemporal signal featuring spatially localized propagating disturbances can be found in \cite{GiannakisEtAl17}.

\section{\label{secBackground}Background}

\subsection{\label{secPrelim}Dynamical system and spaces of observables}

We begin by introducing the dynamical system and the spaces of observables under study. The dynamics evolves by a $C^1 $ flow map  $\Phi^t : X \mapsto X$,  $t \in \mathbb{R}$, on a manifold $X$, possessing a Borel  ergodic invariant probability measure $\mu $ with compact support $A \subseteq X$ (e.g., an attractor). The system develops patterns on a spatial domain $Y$, which has the structure of a compact metric space equipped with a finite Borel measure (volume) $ \nu $. As a natural space of vector-valued observables, we consider the Hilbert space $H = L^2(X,\mu; H_Y)$ of square-integrable functions with respect to the invariant measure $ \mu $, taking values in  $H_Y = L^2(Y,\nu)$. That is, the elements of $H$ are equivalence classes of functions $ \vec f : X \mapsto H_Y $, such that for $ \mu $-almost every dynamical state $x\in X$, $\vec f(x) $ is a scalar (complex-valued) field on $Y$, square-integrable with respect to $\nu$. For every such observable $ \vec f $, the map $t \mapsto \vec f(\Phi^t(x)) $ describes a spatiotemporal pattern generated by the dynamics. Given $\vec f, \vec f' \in H $ and $ g, g' \in H_Y $, the corresponding inner products on $H$ and $H_Y$ are given by $\langle \vec f, \vec f' \rangle_{H} = \int_X \langle \vec f( x ), \vec f'( x ) \rangle_{H_Y} \, d\mu(x)$ and $ \langle g, g' \rangle_{H_Y} = \int_Y g^*(y) g'(y) \, d\nu(y)$, respectively.

An important property of $H$ is that it exhibits the isomorphisms 
\begin{displaymath}
  H \simeq H_X \otimes H_Y \simeq H_\Omega, 
\end{displaymath}
where $H_X = L^2(X,\mu)$ and $H_\Omega = L^2( \Omega, \rho ) $ are Hilbert spaces of  scalar-valued functions on $X$ and the product space $\Omega = X \times Y$, square-integrable with respect to the invariant measure $ \mu $ and the product measure $ \rho = \mu \times \nu $, respectively (the inner products of $H_X$ and $H_\Omega$ have analogous definitions to the inner product of $H_Y$). That is, every $\vec f\in H$ can be equivalently viewed as an element of the tensor product space $H_X\otimes H_Y$, meaning that it can be decomposed as $\vec f = \sum_{j=0}^\infty \varphi_j \otimes \psi_j$ for some $\varphi_j \in H_X$ and $\psi_j \in H_Y$, or it can be represented by a scalar-valued function $f \in H_\Omega$ such that $\vec f(x)(y) = f(x,y)$. Of course, not every observable $ \vec f \in  H $ is of pure tensor-product form, $ \vec f = \varphi \otimes \psi$,  for some $ \varphi \in H_X $ and $ \psi \in H_Y $. 

We consider that measurements $ \vec F(x_n)$ of the system are taken along a dynamical trajectory $ x_n = \Phi^{n\tau}(x_0) $, $ n \in \mathbb{N}$, starting from a point $x_0 \in X$ at a fixed sampling interval $ \tau > 0 $ through a continuous vector-valued observation map $ \vec F \in H$. We assume that $ \tau $ is such that $ \mu $ is an ergodic invariant probability measure of the discrete-time map $ \Phi^\tau $. 

\subsection{\label{secSKernel}Separable data decompositions via scalar kernel eigenfunctions}

Before describing the operator-valued kernel formalism at the core of VSA, we outline the standard approach to separable decompositions of spatiotemporal data as in~\eqref{eqPODDecomp} via eigenfunctions of kernel integral operators associated with scalar-valued kernels. In this context, a kernel is a continuous bivariate function $ k : X \times X \mapsto \mathbb{R } $, which assigns a measure of correlation or similarity to pairs of dynamical states in $ X$. Sometimes, but not always, we will require that $ k $ be symmetric, i.e., $ k( x, x' ) = k( x', x ) $ for all $ x, x' \in X $. Two examples of popular kernels used in applications (both symmetric) are the covariance kernels employed in POD,
\begin{equation}
  \label{eqKCov}
  k( x, x' ) = \langle \vec F( x ) - \bar F,\vec F( x' ) - \bar F \rangle_{H_Y}, \quad \bar F = \int_X \vec F(x) \, d\mu(x),
\end{equation}
and radial Gaussian kernels,
\begin{equation}
  \label{eqKGauss}
  k(x, x' ) = \exp \left( - \frac{ \lVert \vec F( x ) - \vec F( x' ) \rVert_{H_Y}^2 }{ \epsilon } \right), \quad \epsilon > 0,
\end{equation}
which are frequently used in manifold learning applications. Note that in both of the above examples the dependence of $ k( x, x' ) $ on $ x $ and $ x' $ is through the values of $ \vec F $ at these  points alone; this allows $ k(x,x') $ to be computable from observed data, without explicit knowledge of the underlying dynamical states $ x $ and $x'$. Hereafter, we will always work with such ``data-driven'' kernels.  

Associated with every scalar-valued kernel is an integral operator $ K : H_X \mapsto H_X $, acting on $ f \in H_X $ according to the formula
\begin{equation}
  \label{eqKScalar}
  K f ( x ) = \int_X k( x, x' ) f(x')\, d\mu(x').
\end{equation}
If $ k $ is symmetric, then by compactness of $ A $ and continuity of $ k $, $ K $ is a compact, self-adjoint operator with an associated orthonormal basis $ \{ \varphi_0, \varphi_1, \ldots \} $ of $ H_X $ consisting of its eigenfunctions. Moreover, the eigenfunctions $ \varphi_j $ corresponding to nonzero eigenvalues are continuous. These eigenfunctions are employed as the chronos modes in~\eqref{eqPODDecomp}, each inducing a continuous temporal pattern, $ t \mapsto \varphi_j( \Phi^t( x ) ) $, for every state $ x \in X $.  The spatial pattern $ \psi_j \in H_Y $ corresponding to $ \varphi_j $ is obtained by pointwise projection of the observation map onto $ \varphi_j $, namely 
\begin{equation}
    \label{eqProjPsi}
    \psi_j( y ) = \langle \varphi_j, F_y \rangle_{H_X},
\end{equation}
where $ F_y \in H_X $ is the continuous scalar-valued function on $X $ satisfying $ F_y( x ) = \vec F( x )( y ) $ for all $ x \in X $. 
\subsection{\label{secSKernelDelay}Delay-coordinate maps and Koopman operators}

A potential shortcoming of spatiotemporal pattern extraction via the kernels in~\eqref{eqKCov} and~\eqref{eqKGauss} is that the corresponding integral operators depend on the dynamics only indirectly, e.g., through the geometrical structure of the set $\vec F( A ) \subset H_Y $ near which the data is concentrated. Indeed, a well known deficiency of POD, particularly in systems with symmetries, is failure to identify low-variance, yet dynamically important patterns \cite{AubryEtAl93}. As a way of addressing this issue, it has been found effective \cite{BroomheadKing86,VautardGhil89,GhilEtAl02,GiannakisMajda12a,BerryEtAl13} to first embed the observed data in a higher-dimensional data space through the use of delay-coordinate maps, and then extract spatial and temporal patterns through a kernel operating in delay-coordinate space. For instance, analogs of the covariance and Gaussian kernels in~\eqref{eqKCov} and~\eqref{eqKGauss} in delay-coordinate space are given by
\begin{equation}
    \label{eqKQCov}
    k_Q( x, x' ) = \frac{1}{Q} \sum_{q=0}^{Q-1} \langle \vec F( \Phi^{-q \tau}( x ) ) - \bar F, \vec F( \Phi^{-q \tau}( x') ) - \bar F \rangle_{H_Y}, 
\end{equation}
and
\begin{equation}
    \label{eqKQGauss}
    k_Q( x, x' ) = \exp \left( - \frac{1}{\epsilon Q} \sum_{q=0}^{Q-1} \lVert \vec F( \Phi^{-q \tau}( x ) ) - \vec F( \Phi^{-q \tau}( x' ) ) \rVert_{H_Y}^2 \right), 
\end{equation}
respectively, here $ Q \in \mathbb{N} $ is the number of delays. The covariance kernel in~\eqref{eqKQCov} is essentially equivalent to the kernel employed in multi-channel SSA \cite{GhilEtAl02} in an infinite-channel limit, and the Gaussian kernel in~\eqref{eqKQGauss} is closely related to the kernel utilized in NLSA (though the NLSA kernel employs a state-dependent distance scaling akin to~\eqref{eqKQ} ahead, as well as Markov normalization, and these features lead to certain technical advantages compared to unnormalized radial Gaussian kernels). 

As is well known \cite{PackardEtAl80,Takens81,SauerEtAl91,Robinson05,DeyleSugihara11}, delay-coordinate maps can help recover the topological structure of state space from partial measurements of the system (i.e., non-injective observation maps), but in the context of kernel algorithms they also endow the kernels, and thus the corresponding eigenfunctions, with an explicit dependence on the dynamics. In \cite{Giannakis17,DasGiannakis17}, it was established that as the number of delays $ Q $ grows, the integral operators $K_Q$ associated with a family of scalar kernels $k_Q$ operating in delay-coordinate space converge in operator norm, and thus in spectrum, to a compact kernel integral operator $K_\infty$ on $H_X$ commuting with the Koopman evolution operators \cite{BudisicEtAl12,EisnerEtAl15} of the dynamical system. The latter are the unitary operators $U^t : H_X \mapsto H_X$, $ t \in \mathbb{R} $, acting on observables by composition with the flow map, 
\begin{displaymath}
    U^t f = f \circ \Phi^t,
\end{displaymath}
thus governing the evolution of observables in $H_X$ under the dynamics. 

In the setting of measure-preserving ergodic systems, associated with $U^t$ is a distinguished orthonormal set $\{ z_j \}$ of observables $z_j \in H_X$ consisting of Koopman eigenfunctions (see Appendix~\ref{appKoopman}). These observables have the special property of exhibiting time-periodic evolution under the dynamics at a single frequency $ \alpha_j \in \mathbb{R}$ intrinsic to the dynamical system, 
\begin{displaymath}
U^t z_j = e^{i \alpha_j t} z_j,
\end{displaymath}
even if the underlying dynamical flow $ \Phi^t $ is aperiodic. Moreover, every Koopman eigenspaces is one-dimensional by ergodicity. Because commuting operators have common eigenspaces, and the eigenspaces of compact operators corresponding to nonzero eigenvalues are finite-dimensional, it follows that as $Q$ increases, the eigenfunctions of $K_Q $ at nonzero eigenvalues acquire increasingly coherent (periodic or quasiperiodic) time evolution associated with a finite number of Koopman eigenfrequencies $ \alpha_j$. This property significantly enhances the physical interpretability and predictability of these patterns, providing justification for the skill of methods such as SSA and NLSA in extracting dynamically significant patterns from complex systems. Conversely, because kernel integral operators are generally more amenable to approximation from data than Koopman operators (which can have a highly complex spectral behavior, including non-isolated eigenvalues and continuous spectrum), the operators $K_Q$ provide an effective route for identifying finite-dimensional approximation spaces to stably and efficiently solve the Koopman eigenvalue problem. 

\subsection{\label{secCovGauss}Differences between covariance and Gaussian kernels}

Before closing this section, it is worthwhile pointing out two differences between covariance and Gaussian kernels, indicating that the latter may be preferable to the former in applications.

First, Gaussian kernels are strictly positive, and bounded below on compact sets. That is, for every compact set $S \subseteq X$ (including $S=A$), there exists a constant $c_S > 0 $ such that $ k( x, x' ) \geq c_S $ for all $ x,x'\in S$. This property allows Gaussian kernels to be normalizable to ergodic Markov diffusion kernels \cite{CoifmanLafon06,BerrySauer16b}. In a dynamical systems context, an important  property of such kernels is that the corresponding integral operators always have an eigenspace at eigenvalue 1 containing constant functions, which turns out to be useful in establishing well-posedness of Galerkin approximation techniques for Koopman eigenfunctions \cite{DasGiannakis17}. Markov diffusion operators are also useful for constructing spaces of observables of higher regularity than $L^2$, such as Sobolev spaces.
 
Second, if there exists a finite-dimensional linear subspace of $H_Y$ containing the image of $A$  under $ \vec F $,  then the integral operator $K$ associated with the covariance kernel has necessarily finite rank (bounded above by the dimension of that subspace), even if $ \vec F $ is an injective map on $A$. This effectively limits the richness of observables that can be stably extracted from data-driven approximations to covariance eigenfunctions. In fact, it is a well known property of covariance kernels that every eigenfunction $ \varphi_j$ at nonzero corresponding eigenvalue depends linearly on the observation map; specifically, up to proportionality constants, $ \varphi_j( x ) = \langle \psi_j, \vec F( x  ) \rangle_{H_Y} $,  and the number of such patterns is clearly finite if $ \vec F(x ) $ spans a finite-dimensional linear space as $x$ is varied. On the other hand, apart from trivial cases, the kernel integral operators associated with Gaussian kernels have infinite rank, even if $ \vec F $ is non-injective, and moreover if $ \vec F $ is injective they have no zero eigenvalues. In the latter case, data-driven approximations to the eigenfunctions of $K$ provide an orthonormal basis for the full $H_X $ space.  Similar arguments also motivate the use of Gaussian kernels over polynomial kernels. In effect, by invoking the Taylor series expansion  of the exponential function, a Gaussian kernel can be thought of as an ``infinite-order'' polynomial kernel.

\section{\label{secVSA}Vector-valued spectral analysis (VSA) formalism}

The main goal of VSA is to construct a decomposition of the observation map $ \vec F$ via an expansion of the form 
\begin{equation}
    \label{eqVSADecomp}
    \vec F \approx \vec F_l = \sum_{j=0}^l c_j \vec \phi_j, 
\end{equation}
where the $ c_j $ and $ \vec \phi_j $ are real-valued coefficients and vector-valued observables in $H $, respectively. Along a dynamical trajectory starting at $x\in X$, every $\vec \phi_j$ gives rise to a spatiotemporal pattern $t \mapsto \vec\phi_j(\Phi^t(x))$, generalizing the time series $ t \mapsto \varphi_j(\Phi^t(x)) $ from Section~\ref{secSKernel}. A key consideration in the VSA construction is that the  recovered patterns should not necessarily be of the form $ \vec \phi_j =  \varphi_j \otimes \psi_j $ for some $ \varphi_j \in H_X $ and $ \psi_j \in H_Y $, as would be the case in the conventional decomposition in~\eqref{eqPODDecomp}. To that end, we will determine the $ \vec \phi_j$ through the vector-valued eigenfunctions of an integral operator acting on $H$ directly, as opposed to first identifying scalar-valued eigenfunctions in $H_X$, and then forming tensor products with the corresponding projection-based spatial patterns, as in Section~\ref{secSKernel}. As will be described in detail below, the integral operator nominally employed by VSA is constructed using the theory of operator-valued kernels \cite{MicchelliPontil05,CaponnettoEtAl08,CarmeliEtAl10}, combined with delay-coordinate maps and Markov normalization as in NLSA.   

\subsection{\label{secVSAKernel}Operator-valued kernel and vector-valued eigenfunctions}

Let $B(H_Y)$ be the Banach space of bounded linear maps on $H_Y$, equipped with the operator norm. For our purposes, an operator-valued kernel is a continuous map $ l : X \times X \mapsto B(H_Y) $, mapping pairs of dynamical states in $ X $ to a bounded operator on $H_Y$. Every such kernel has an associated integral operator $ L : H \mapsto H $, acting on vector-valued observables according to the formula (cf.~\eqref{eqKScalar})
\begin{displaymath}
    L \vec f(x) = \int_X l( x, x' ) \vec f(x') \, d\mu(x'),
\end{displaymath}
where the integral above is a Bochner integral (a vector-valued generalization of the Lebesgue integral). Note that operator-valued kernels and their corresponding integral operators can be viewed as generalizations of their scalar-valued counterparts from Section~\ref{secSKernel}, in the sense that if $Y$ only contains a single point, then $H_Y$ is isomorphic to the vector space of complex numbers (equipped with the standard operations of addition and scalar multiplication and the inner product $ \langle w, z \rangle_{\mathbb{C}} = w^* z$), and $B(H_Y)$ is isomorphic to the space of multiplication operators on $\mathbb{C}$ by complex numbers. In that case, the action $l(x,x') \vec f(x')$ of the linear map $ l(x,x') \in B(H_Y) $ on the function $ \vec f(x') \in H_Y$ becomes equivalent to multiplication of the complex number $ f(x) $, where $ f $ is a complex-valued observable in $H_X$, by the value $k(x,x') \in \mathbb{C} $ of a scalar-valued kernel $k $ on $X$. 

Consider now an operator-valued kernel $ l : X \times X \mapsto B(H_Y)$, such that for every pair $ ( x, x' ) $ of states in $X$, $ l( x, x' ) =L_{xx'} $ is a kernel integral operator on $H_Y $ associated with a continuous kernel $ l_{xx'} : Y \times Y \mapsto \mathbb{R} $ with the symmetry property
\begin{equation}
    l_{xx'}(y,y') = l_{x'x}(y',y), \quad \forall x,x' \in X, \quad \forall y,y' \in Y.
    \label{eqLSym}
\end{equation}
This operator acts on a scalar-valued function  $ g \in H_Y $ on the spatial domain via an integral formula analogous to~\eqref{eqKScalar}, viz.    
\begin{displaymath}
    L_{xx'} g( y ) = \int_Y l_{xx'}( y, y' ) g( y' ) \, d\nu(y' ). 
\end{displaymath} 
Moreover, it follows from~\eqref{eqLSym} that the corresponding operator $ L $ on vector-valued observables is self-adjoint and compact, and thus there exists an orthonormal basis $ \{ \vec \phi_j\} $ of $ H $ consisting of its eigenfunctions,
\begin{displaymath}
    L \vec \phi_j = \lambda_j \vec \phi_j, \quad \lambda_j \in \mathbb{R}. 
\end{displaymath}
Hereafter, we will always order the eigenvalues $ \lambda_j$ of integral operators in decreasing order starting at $ j = 0$. By continuity of $ l $ and $ l_{xx'} $, every eigenfunction $ \vec \phi_j $ at nonzero corresponding eigenvalue  is a continuous function on $ X $, taking values in the space of continuous functions on $Y$. Such eigenfunctions can be employed in the VSA decomposition in~\eqref{eqVSADecomp} with the expansion coefficients 
\begin{equation}
    \label{eqCJVSA}
    c_j = \langle \vec \phi_j, \vec F \rangle_H = \int_X \langle \vec\phi_j( x ), \vec F( x ) \rangle_{H_Y} \, d\mu(x). 
\end{equation}
Note that, as with scalar kernel techniques, the decomposition in~\eqref{eqVSADecomp} does not include eigenfunctions at zero corresponding eigenvalue, for, to our knowledge, no data-driven approximation schemes are available for such eigenfunctions. See Section~\ref{secDataDriven} and Appendix~\ref{appDataDriven} for further details.  

Because $ H $ is isomorphic as  Hilbert space to the space $H_\Omega$ of scalar-valued observables on the product space $ \Omega = X \times Y $ (see Section~\ref{secPrelim}), every operator-valued kernel satisfying~\eqref{eqLSym}  can be constructed from a symmetric scalar kernel $ k : \Omega \times \Omega \mapsto \mathbb{R}$ by defining $ l( x, x' ) =  L_{xx'} $ as the integral operator associated with the kernel 
\begin{equation}
    \label{eqVecScalOp}
    l_{xx'}(y,y') = k(\omega, \omega'), \quad  \omega = (x,y), \quad \omega' = (x',y').
\end{equation}
In particular, the vector-valued eigenfunctions of $L$ are in one-to-one correspondence with the scalar-valued eigenfunctions of the integral operator $ K : H_\Omega \mapsto H_\Omega $ associated with $k$, where 
\begin{equation}
    K f( \omega ) = \int_\Omega k(\omega,\omega') f(\omega') \, d\rho(\omega').
    \label{eqKOmega}
\end{equation}
That is, the eigenvalues and eigenvectors of $ K $ satisfy the equation $ K \phi_j =\lambda_j \phi_j $ for the same eigenvalues as those of $L$, and we also have
\begin{equation}
    \label{eqVecScal}
    \vec \phi_j(x)(y) = \phi_j( ( x, y ) ), \quad \forall x \in X, \quad \forall y \in Y. 
\end{equation}
It is important to note that unless $ k $ is separable as a product of kernels on $ X $ and $ Y $, i.e., $ k( ( x, y ), ( x', y' ) ) = k^{(X)}(x,x') k^{(Y)}(y,y') $ for some $k^{(X)} : X \times X \mapsto \mathbb{R} $ and $ k^{(Y)} : Y \times Y \mapsto \mathbb{R} $, the $ \vec \phi_j $ will not be of pure tensor product form, $ \vec \phi_j = \varphi_j \otimes \psi_j $ with $ \varphi_j \in H_X $ and $ \psi_j \in H_Y $. Thus, passing to an operator-valued kernel formalism allows one to perform decompositions of significantly higher generality than the conventional approach in~\eqref{eqPODDecomp}. 

\subsection{\label{secKernelDelays}Operator-valued kernels with delay-coordinate maps}

While the framework described in Section~\ref{secVSAKernel} can be implemented with a very broad range of kernels, VSA employs kernels leveraging the insights gained from SSA, NLSA, and related techniques on the use of kernels operating in delay-coordinate space. That is, analogously to the kernels employed by these methods that depend on the values $ \vec F( (x) ),\vec F( \Phi^{- \tau}(x) ), \ldots, \vec F( \Phi^{-(Q-1) \tau}(x)) $ of the observation map on dynamical trajectories, VSA is based on kernels on the product space $ \Omega$ that also depend on data observed on dynamical trajectories, but with the key difference that this dependence is through the \emph{local} values  $  F_y( x ), F_y( \Phi^{- \tau}(x) ), \ldots,  F_y( \Phi^{-(Q-1) \tau}(x)) $ of the observation map at each point $ y $ in the spatial domain $ Y $. Specifically, defining the family of pointwise delay embedding maps $ \tilde F_Q : \Omega \mapsto \mathbb{R}^Q $ with $ Q \in \mathbb{N} $ and 
\begin{equation}
    \label{eqFQ}
    \tilde F_Q((x,y)) = \left( F_y( x ), F_y( \Phi^{- \tau}(x) ), \ldots,  F_y( \Phi^{-(Q-1) \tau}(x)) \right),
\end{equation}
we require that the kernels $ k_Q : \Omega \times \Omega \mapsto \mathbb{R} $ utilized in VSA have the following properties: 
\begin{enumerate}
    \item For every $Q \in \mathbb{N} $, $ k_Q $ is the pullback under $ \tilde F_Q $ of a continuous kernel $ \tilde k_Q : \mathbb{R}^Q \times \mathbb{R}^Q \mapsto \mathbb{R} $, i.e.,  
    \begin{equation}
        \label{eqKPullback}
        k_Q( \omega, \omega' ) = \tilde k_Q( \tilde F_Q( \omega ), \tilde F_Q( \omega' ) ), \quad \forall \omega,\omega'\in \Omega.
    \end{equation}
\item The sequence of kernels $ k_1, k_2, \ldots $ converges in $ H_\Omega \otimes H_\Omega $ norm to a kernel $ k_\infty \in  H_\Omega \otimes H_\Omega$.
    \item The limit kernel $k_\infty$ is invariant under the dynamics, in the sense that for all $ t\in \mathbb{R}$ and $ ( \rho \times \rho ) $-a.e.\ $ ( \omega, \omega' ) \in \Omega \times \Omega $, where $ \omega = ( x, y ) $ and $ \omega' = ( x', y' ) $,
    \begin{equation}
        \label{eqKInv}
        k_\infty( ( \Phi^t(x), y ), ( \Phi^t(x'), y' ) ) = k_\infty( \omega, \omega' ).
    \end{equation}
\end{enumerate}
We denote the corresponding integral operator on vector-valued observables in $H$, determined through~\eqref{eqVecScalOp}, by $ L_Q$.  As we will see below, operators of this class can be highly advantageous for the analysis of signals with an intermittent spatiotemporal character, as well as signals generated in the presence of dynamical symmetries. In addition, the family $L_Q$ exhibits a commutativity with Koopman operators in the infinite-delay limit as in the case of SSA and NLSA. 

Let $ \omega = ( x, y ) $ and $ \omega' = (x',y') $ with $x,x' \in X $ and $ y, y' \in Y$ be arbitrary points in $ \Omega$. As concrete examples of kernels satisfying the conditions listed above,
\begin{equation}
    \label{eqKCovVSA}
    k_Q(\omega, \omega') = \frac{1}{Q} \sum_{q=0}^{Q-1} \left[ F_y( \Phi^{-q\tau}(x)) - \bar F_y \right] \left[ F_{y'}( \Phi^{-q\tau}(x')) - \bar F_{y'} \right], \quad \bar F_y = \int_X F_y(x) \, d\mu(x), 
\end{equation}
and 
\begin{equation}
    \label{eqKGaussVSA}
    k_Q(\omega,\omega') = \exp \left( - \frac{1}{\epsilon Q} \sum_{q=0}^{Q-1} \left\lvert F_y( \Phi^{-q \tau}(x)) - F_{y'}(\Phi^{-q \tau}(x')) \right\rvert^2 \right), \quad \epsilon>0,
\end{equation}
are analogs of the covariance and Gaussian kernels in \eqref{eqKCov} and~\eqref{eqKGauss}, respectively, defined on $ \Omega$. For the reasons stated in Section~\ref{secCovGauss}, in practice we generally prefer working with Gaussian kernels than covariance kernels. Moreover, following the approach employed in NLSA and in \cite{BerryHarlim16,BerryEtAl15,GiannakisEtAl15,Giannakis17}, we nominally consider a more general class of Gaussian kernels than~\eqref{eqKGaussVSA}, namely
\begin{equation}
    \label{eqKQ}
    k_Q(\omega,\omega') = \exp \left( - \frac{s_Q(\omega) s_Q(\omega')}{\epsilon Q} \sum_{q=0}^{Q-1} \left\lvert F_y( \Phi^{q \tau}(x)) - F_{y'}(\Phi^{q \tau}(x')) \right\rvert^2 \right), \quad \epsilon>0,
   \end{equation}
   where  $ s_Q : \Omega \mapsto \mathbb{R}$ is a continuous non-negative scaling function.  Intuitively, the role of $s_Q $ is to adjust the bandwidth (variance) of the Gaussian kernel in order to account for variations in the sampling density and time tendency of the data. The explicit construction of this function is described in Appendix~\ref{appBandwidth}. For the purposes of the present discussion, it suffices to note that $ s_Q( \omega ) $ can be evaluated given the values of $F_y $ on the lagged trajectory $ \Phi^{-q\tau}(x) $,  so that, as with the covariance and radial Gaussian kernels, the class of kernels in~\eqref{eqKQ} also satisfy~\eqref{eqKPullback}. The existence of limit $ k_\infty $ for this family of kernels, as well as the covariance kernels in~\eqref{eqKCovVSA}, satisfying the conditions listed above is established in Appendix~\ref{appKDelay}.
   
\subsection{\label{secMarkov}Markov normalization}

As a final kernel construction step, when working with a strictly positive, symmetric  kernel $ k_Q$, such as~\eqref{eqKGaussVSA} and \eqref{eqKQ}, we normalize it to a continuous Markov kernel $p_Q : \Omega \times \Omega \mapsto \mathbb{R}$, satisfying $ \int_\Omega p_Q( \omega, \cdot ) \, d\rho = 1$ for all $\omega \in \Omega$, using the normalization procedure introduced in the diffusion maps algorithm \cite{CoifmanLafon06} and in~\cite{BerrySauer16b}; see Appendix~\ref{appMarkov} for a description. Due to this normalization, the corresponding integral operator $ P_Q : H_\Omega \mapsto H_\Omega$ is an ergodic Markov operator having a simple eigenvalue $ \lambda_0 = 1 $ and a corresponding constant eigenfunction $ \phi_0$. Moreover, the range of $P_Q$ is included in the space of continuous functions of $ \Omega$. While this operator is not necessarily self-adjoint (since the kernel $p_Q$ resulting from diffusion maps normalization is generally non-symmetric), it can be shown that it is related to a self-adjoint, compact operator by a similarity transformation. As a result, all eigenvalues of $P_Q $ are real, and admit the ordering $ 1 = \lambda_0 > \lambda_1 \geq \lambda_2 \cdots $. Moreover, there exists a (non-orthogonal) basis $ \{ \phi_0, \phi_1, \ldots \} $ of $H_\Omega $ consisting of eigenfunctions  corresponding to these eigenvalues, as well as a dual basis $ \{ \phi'_0, \phi'_1, \ldots \} $ consisting of eigenfunctions of $ P^*_Q $ satisfying $ \langle \phi'_i, \phi_j \rangle_{H_\Omega} = \delta_{ij}$. As with their unnormalized counterparts $k_Q$, the sequence of Markov kernels $ p_Q $ has a well-defined limit $p_\infty \in H_\Omega \otimes H_\Omega$ as $ Q \to\infty$; see Appendix~\ref{appMarkov} for further details. 
   
The eigenfunctions $ \phi_j $ induce vector-valued observables $ \vec \phi_j \in H$ through~\eqref{eqVecScal}, which are in turn eigenfunctions of an integral operator $ \mathcal{P}_Q : H \mapsto H$ associated with the operator-valued kernel determined via~\eqref{eqVecScalOp} applied to the Markov kernel $p_Q$. Similarly, the dual eigenfunctions $ \phi'_j$ induce vector-valued observables $ \vec \phi'_j \in H $, which are eigenfunctions of $\mathcal{P}^*_Q$ satisfying $ \langle \vec \phi_i, \vec \phi_j \rangle_{H} = \delta_{ij} $. Equipped with these observables, we perform the VSA decomposition in~\eqref{eqVSADecomp} with the expansion coefficients $c_j = \langle \vec \phi_j', \vec F \rangle_H$. The latter expression can be viewed as a generalization of~\eqref{eqCJVSA}, applicable for non-orthonormal eigenbases.

\section{\label{secVSAProperties}Properties of the VSA decomposition}

In this section, we study the properties of the operators $K_Q$ employed in VSA and their eigenfunctions in two relevant scenarios in spatiotemporal data analysis, namely data generated by systems with (i) dynamical symmetries, and (ii) non-trivial Koopman eigenfunctions. These topics will be discussed in Sections~\ref{secSymmetries} and~\ref{secKoopman}, respectively. We begin in Section~\ref{secBundle} with some general observations on the topological structure of spatiotemporal data in delay-coordinate space, and the properties this structure imparts on the recovered eigenfunctions.     

\subsection{\label{secBundle}Bundle structure of spatiotemporal data}

In order to gain insight on the behavior of VSA, it is useful to consider the triplet $(\Omega,B_Q,\pi_Q)$, where $B_Q = \tilde F_Q( \Omega) $ is the image of the product space $ \Omega$ under the delay-coordinate observation map, and $ \pi_Q : \Omega \mapsto B_Q $ is the continuous surjective map defined as $\pi_Q( \omega ) = \tilde F_Q( \omega)$ for any $ \omega \in \Omega $. Such a triplet forms a topological bundle with $ \Omega $, $ B_Q $, and $ \pi_Q $ playing the role of the total space, base space, and projection map, respectively. In particular, $\pi_Q$ partitions $\Omega$ into equivalence classes 
\begin{equation}
  \label{eqEquivQ}
  [\omega]_Q = \pi_Q^{-1}(x) \subseteq \Omega, 
\end{equation}
called fibers, on which $\pi_Q(\omega) $ attains a fixed value (i.e.,  $ \tilde \omega $ lies in $ [ \omega ]_Q $ if $ \pi_Q( \tilde \omega ) = \pi_Q( \omega ) $). 

By virtue of~\eqref{eqKPullback}, the kernel $ k_Q $ is a continuous function, constant on the $[\cdot]_Q$ equivalence classes, i.e., for all $ \omega,\omega' \in \Omega$, $ \tilde \omega \in [\omega]_Q$, and $ \tilde \omega' \in [\omega']_Q$,
\begin{equation}
    \label{eqKQEquiv}
    k_Q(\omega,\omega') = k_Q( \tilde \omega, \tilde \omega').
\end{equation}
Therefore, since for any $ f \in H_\Omega$ and $ \tilde \omega \in [\omega]_Q$, 
\begin{displaymath}
    K_Q f( \omega ) = \int_\Omega k_Q( \omega, \omega' ) f(\omega') \, d\rho(\omega' ) = \int_\Omega k_Q( \tilde \omega, \omega' ) f(\omega') \, d\rho(\omega' ) = K_Q f( \tilde \omega ),,
\end{displaymath}
the range of the integral operator $ K_Q $ is a subspace of the continuous functions on $ \Omega$, containing constant functions on the $ [ \cdot ]_Q $ equivalence classes. Correspondingly, the eigenfunctions $ \phi_j $ corresponding to nonzero eigenvalues (which lie in $ \ran K_Q$) have the form $ \phi_j = \eta_j \circ \pi_Q $, where $ \eta_j $ are continuous functions in the Hilbert space $  L^2( B_Q, \pi_{Q*} \rho) $ of scalar-valued functions on $B_Q$, square-integrable with respect to the pushforward of the measure $ \rho $ under $ \pi_Q$. We can thus conclude that, viewed as a scalar-valued function on $\Omega$, the VSA-reconstructed signal $ \vec F_l $ from~\eqref{eqVSADecomp} lies in the closed subspace $ \overline{ \ran K_Q } = \overline{ \spn\{ \phi_j : \lambda_j > 0 } \} $ of $ H_\Omega $ spanned by constant functions on the $ [ \cdot ]_Q $ equivalence classes. Note that  $ \overline{ \ran K_Q } $ is not necessarily decomposable as a tensor product of $H_X$ and $H_Y$ subspaces.

Observe now that with the definition of the kernel in~\eqref{eqKQ}, the $[\cdot]_Q$ equivalence classes consist of pairs of dynamical states $x \in \Omega$ and spatial points $y \in Y$ for which the evolution of the observable $ F_y$ is identical over $Q$ delays. While one can certainly envision scenarios where these equivalence classes each contain only one point, in a number of cases of interest, including the presence of dynamical symmetries examined below, the $[\cdot]_Q$ equivalence classes will be nontrivial, and as a result $\mathcal{H}_Q$ will be a strict subspace of $H_\Omega$. In such cases, the patterns recovered by VSA naturally factor out data redundancies, which generally enhances both robustness and physical interpretability of the results. Besides spatiotemporal data, the bundle construction described above may be useful in other scenarios, e.g., analysis of data generated by dynamical systems with varying parameters \cite{YairEtAl17}.

\subsection{\label{secSymmetries}Dynamical symmetries}

An important class of spatiotemporal systems exhibiting nontrivial $[\cdot]_Q$ equivalence classes is PDE models equivariant under the action of symmetry groups on the spatial domain \cite{HolmesEtAl96}. As a concrete example, we consider a PDE for a scalar field in $H_Y$, possessing a $C^1$ inertial manifold; i.e., a finite-dimensional, forward-invariant submanifold of $H_Y$ containing the attractor of the system, and onto which every trajectory is exponentially attracted \cite{ConstantinEtAl89}. In this setting, the inertial manifold plays the role of the state space manifold $X$. Moreover, we assume that the full system state is observed, so that the observation map $ \vec F $  reduces to the inclusion  $ X \hookrightarrow  H_Y $. 

Consider now a topological group $G$ (the symmetry group) with a continuous left action  $\Gamma_Y^g: Y \mapsto Y $, $ g \in G $, on the spatial domain, preserving null sets with respect to $\nu$. Suppose also that the dynamics is equivariant under the corresponding induced action $\Gamma_X^g : X \mapsto X $, $ \Gamma_X^g( x ) = x \circ \Gamma_Y^{g^{-1}} $, on the state space manifold. This means that the dynamical flow map and the symmetry group action commute, 
\begin{equation}
  \label{eqSymmetry}
  \Gamma^g_X \circ \Phi^t = \Phi^t \circ \Gamma_X^g, \quad \forall t \in \mathbb{R}, \quad \forall g \in G,
\end{equation}
or, in other words, if $ t \mapsto \Phi^t( x ) $ is a solution starting at $ x \in X $, then $ t \mapsto \Phi^t( \Gamma^g_X( x ) ) $ is a solution starting at $ \Gamma^g_X( x ) $. Additional aspects of symmetry group actions and equivariance are outlined in Appendix~\ref{appSymmetries}. Our goal for this section is to examine the implications of~\eqref{eqSymmetry} to the properties of the operators $K_Q$ employed in VSA and their eigenfunctions.

\subsubsection{\label{secEigSym}Dynamical symmetries and VSA eigenfunctions}

We begin by considering the induced action $ \Gamma_\Omega^g : \Omega \mapsto \Omega $ of $ G $ on the product space $ \Omega $, defined as
\begin{displaymath}
  \Gamma_\Omega^g = \Gamma_X^g \otimes \Gamma_Y^g.  
\end{displaymath}
This group action partitions $ \Omega $ into orbits, defined for every $ \omega \in \Omega $ as the subsets $ \Gamma_\Omega( \omega ) \subseteq \Omega $ with
\begin{displaymath}
  \Gamma_\Omega( \omega ) = \{ \Gamma^g_\Omega( \omega ) \mid g \in G \}.
\end{displaymath}
As with the subsets $ [ \omega ]_Q \subset \Omega $ from~\eqref{eqEquivQ} associated with delay-coordinate maps, the $ G $-orbits on $ \Omega$ form equivalence classes, consisting of points linked together by symmetry group actions (as opposed to having common values under delay coordinate maps). In general, these two sets of equivalence classes are unrelated, but in the presence of dynamical symmetries, they are, in fact, compatible, as follows:

\begin{prop}
    \label{propEquiv}
    If the equivariance property in~\eqref{eqSymmetry} holds, then for every $ \omega \in \Omega$, the $ G $-orbit $ \Gamma_\Omega( \omega ) $ is a subset of the  $ [ \omega ]_Q $ equivalence class. As a result, the following diagram commutes:
    \begin{displaymath}
       \begin{tikzcd}
           \Omega \arrow{r}{\Gamma^g_\Omega} \arrow{d}{\pi_Q} & \Omega \arrow{ld}{\pi_Q} \\
           B_Q
       \end{tikzcd}
    \end{displaymath}
\end{prop}

\begin{proof}
    Let $x(y)$ denote the value of the dynamical state $ x \in X \subset H_Y $ at $ y \in Y $.  It follows from~\eqref{eqSymmetry} that for every $ t \in \mathbb{R}$, $ g \in G$, $x \in X$, and $ y \in G $,
    \begin{displaymath}
        \Phi^t(\Gamma^g_X(x))( \Gamma^g_Y(y)) = \Gamma^g_X( \Phi^t(x))(\Gamma^g_Y(y)) = \Gamma_X^{g^{-1}}( \Gamma^g_X(\Phi^t(x))) = \Phi^t(x)(y).
    \end{displaymath}
    Therefore, since $ \vec F $ is an inclusion, for every $ g \in G$ and  $ \omega = (x, y ) \in \Omega $, we have
    \begin{align*}
        \tilde F_Q( \omega ) &= \left(  F_y( x ), F_y( \Phi^{-\tau}(x)), \ldots, F_y( \Phi^{-(Q-1) \tau} (x))  \right) \\
        &= \left(  x( y ), \Phi^{-\tau}(x)(y), \ldots, \Phi^{-(Q-1)\tau}(x)(y) \right) \\
        &= \left(  \Gamma^g_X(x)( \Gamma^g_Y(y) ), \Phi^{-\tau}(\Gamma^g_X(x))(\Gamma^g_Y(y)), \ldots, \Phi^{-(Q-1)\tau}(\Gamma^g_X(x))(\Gamma^g_Y(y)) \right) \\
        &= \tilde F_Q( \Gamma^g_\Omega(\omega)). \qedhere
    \end{align*}
\end{proof}

We therefore conclude from Proposition~\ref{propEquiv} and~\eqref{eqKQEquiv} that the kernel $k_Q $ is constant on $ G $-orbits,     
\begin{equation}
  \label{eqKSym}
  k_Q( \Gamma_\Omega^g( \omega ), \Gamma_\Omega^{g'}( \omega' ) ) = k_Q(\omega, \omega'), \quad \forall\omega,\omega' \in \Omega, \quad \forall g,g' \in G, 
\end{equation}
and therefore the eigenfunctions $ \phi_j$ corresponding to nonzero eigenvalues of $K_Q$ are continuous functions with the invariance property
\begin{displaymath}
    \phi_j \circ \Gamma^g_\Omega = \phi_j, \quad \forall g \in G.
\end{displaymath}
This is one of the key properties of VSA, which we interpret as factoring the symmetry group from the recovered spatiotemporal patterns.     

\subsubsection{\label{secGroupSpec}Spectral characterization}

In order to be able to say more about the implication of the results in Section~\ref{secEigSym} at the level of operators, we now assume that the group action  $ \Gamma^g_\Omega $ preserves the measure $ \rho$. Then, there exists a unitary representation of $ G $ on $ H_\Omega $, whose representatives are unitary operators $ R^g_\Omega : H_\Omega \mapsto H_\Omega $ act on functions $ f \in H_\Omega$ by composition with $ \Gamma^g_\Omega$, i.e., $ R^g_\Omega f = f \circ \Gamma_\Omega^g $. Another group of unitary operators acting on $ H_\Omega $ consists of the Koopman operators, $ \tilde U^t : H_\Omega \mapsto H_\Omega$, which we define here via a trivial lift of the Koopman operators $U^t $ on $H_X$, namely $\tilde U^t = U^t \otimes I_{H_Y}$, where $I_{H_Y} $ is the identity operator on $H_Y$; see Appendix~\ref{appKoopman} for further details. In fact, the map $ t \mapsto \tilde U^t $ constitutes a unitary representation of the Abelian group of real numbers (playing the role of time), equipped with addition as the group operation, much like $ g \mapsto R^g_\Omega$ is a unitary representation of the symmetry group $G$.    The following theorem summarizes the relationship between the symmetry group representatives  and the Koopman and kernel integral operators on $H_\Omega$.

\begin{thm}
    \label{thmSym}
    For every $ g \in G$ and $t \in \mathbb{R}$, the operator $ R^g_\Omega$ commutes with $ K_Q$ and $ \tilde U^t$. Moreover, every function in the range of $K_Q$ is invariant under $R^g_\Omega$, i.e., $ R^g_\Omega K_Q = K_Q$.  
\end{thm}

\begin{proof}
    The commutativity between $R^g_\Omega$ and $ \tilde U^t$ is a direct consequence of~\eqref{eqSymmetry}. To verify the claims involving $ K_Q$, we use~\eqref{eqKSym} and the fact that $ \Gamma^g_\Omega$ preserves $ \rho$ to compute
\begin{align*}
  K_Q f( \omega ) &= \int_\Omega k_Q( \omega, \omega' ) f( \omega' ) \, d\rho( \omega' ) \\
  &= \int_\Omega k_Q( \omega, \Gamma_\Omega^{g}( \omega' ) ) f( \Gamma_\Omega^{g}( \omega' ) ) \, d \rho( \omega' ) \\
  &= \int_\Omega k_Q( \Gamma_\Omega^{g^{-1}}( \omega ),   \omega' ) f( \Gamma_\Omega^g( \omega' ) ) \, d \rho( \omega' ) \\
& = \int_\Omega k_Q( \Gamma_\Omega^{g'}( \omega ),   \omega' ) f( \Gamma_\Omega^g( \omega' ) ) \, d \rho( \omega' ) \\
  &= R^{g'}_\Omega K_Q R^g_\Omega f( \omega),
\end{align*}
where $ g $ and $ g' $ are arbitrary, and the equalities hold for $ \rho$-a.e.\ $ \omega \in \Omega $. Setting $ g' = g^{-1} $ in the above, and acting on both sides by $ R^g_\Omega $, leads to $ R^g_\Omega K_Q = K_Q R^g_\Omega $; i.e., $   [ R^g_\Omega, K_Q ] = 0 $, as claimed.  On the other hand, setting $ g $ to the identity element of $ G $ leads to $ K_Q = R^{g'}_\Omega K_Q $, completing the proof of the theorem.\end{proof}

Because commuting operators have common eigenspaces, Theorem~\ref{thmSym} establishes the existence of two sets of common eigenspaces associated with the symmetry group, namely common eigenspaces between $ \Gamma^g_\Omega $ and $ K_Q$ and those between $ R^g_\Omega $ and $\tilde U^t$. In general, these two families of eigenspaces are not compatible since $ \tilde U^t $ and $ K_Q $ many not commute, so for now we will focus on the common eigenspaces between $ R^g_\Omega $ and $ K_Q $ which are accessible via VSA with finitely many delays. In particular, because $ R^g_\Omega K_Q = K_Q $, and every eigenspace $ W_l $ of $ K_Q $ at nonzero corresponding eigenvalue $ \lambda_l $ is finite-dimensional (by compactness of that operator), we can conclude that the $ W_l$ are finite-dimensional subspaces onto which the action of $R^g_\Omega$ reduces to the identity. In other words, the eigenspaces of $ K_Q$ at nonzero corresponding eigenvalues are finite-dimensional trivial representation spaces of $G$, and every VSA eigenfunction $ \phi_j $ is also an eigenfunction of $R^g_\Omega$ at eigenvalue 1. 

At this point, one might naturally ask to what extent these properties are shared in common between VSA and conventional eigendecomposition techniques based on scalar kernels on $ X $. In particular, in the measure-preserving setting for the product measure  $ \rho= \mu \times \nu$ examined above, it must necessarily be the case that the group actions $ \Gamma^g_X$ and $ \Gamma^g_Y$ separately preserve $ \mu $ and $ \nu $, respectively, thus inducing unitary operators $ \Gamma^g_X : H_X \mapsto H_X $ and $ \Gamma^g_Y : H_Y \mapsto H_Y$ defined analogously to $R^g_\Omega$. For a variety of kernels $ k^{(X)}_Q : X \times X \mapsto \mathbb{R}$ that only depend on observed data through inner products and norms on $H_Y$ (e.g., the covariance and Gaussian kernels in Section~\ref{secSKernel}), this implies that the invariance property 
\begin{equation}
    \label{eqKXSym}
    k^{(X)}_Q( \Gamma^g_X( x ), \Gamma^g_X(x') ) = k^{(X)}_Q(x,x')
\end{equation}
holds for all $ g \in G$ and $x,x' \in X$. Moreover, proceeding analogously to the proof of Theorem~\ref{thmSym}, one can show that $ \Gamma^g_X $ and the integral operator $ K^{(X)}_Q : H_X \mapsto H_X$ associated with $k_Q^{(X)}$ commute, and thus have common eigenspaces $W^{(X)}_l$, $ \lambda_l^{(X)} \neq 0 $, which are finite-dimensional invariant subspaces under $ R^g_X$. Projecting the observation map $ \vec F $ onto $ W^{(X)}_l$ as in~\eqref{eqProjPsi}, then yields a finite-dimensional subspace $W^{(Y)}_l \subset H_Y$, which is invariant under $R^g_Y $, and thus $ W^{(X)}_l \otimes W^{(Y)}_l \subset H_\Omega$ is invariant under $ R^g_\Omega$. The fundamental difference between the representation of $G$ on  $ W^{(X)}_l \otimes W^{(Y)}_l \subset H_\Omega$ and that on the $W_l$ subspaces recovered by VSA, is that the former is generally \emph{not} trivial, i.e., in general, $ R^g_\Omega $ does not reduce to the identity map on $ W^{(X)}_l \otimes W^{(Y)}_l$. A well-known consequence of this is that the corresponding spatiotemporal patterns $ \varphi_j \otimes \psi_j$ from~\eqref{eqPODDecomp} become pure symmetry modes (e.g., Fourier modes in dynamical systems with translation invariance), hampering their physical interpretability. 

This difference between VSA and conventional eigendecomposition techniques can be traced back to the fact that that on $X$ there is no analog of Proposition~\ref{propEquiv} relating equivalence classes of points with respect to delay-coordinate maps and group orbits on that space. Indeed, Proposition~\ref{propEquiv} plays an essential role in establishing the kernel invariance property in~\eqref{eqKSym}, which is stronger than~\eqref{eqKXSym} as it allows action by two independent group elements. Equation~\ref{eqKSym} is in turn necessary to determine that $K_Q R^g_\Omega = K_Q $ in Theorem~\ref{thmSym}. In summary, these considerations highlight the importance of taking into account the bundle structure of spatiotemporal data when dealing with systems with dynamical symmetries.   

\subsection{\label{secKoopman}Correspondence with Koopman operators}

\subsubsection{\label{secInfDelay}Behavior of kernel integral operators in the infinite-delay limit}

As discussed in Section~\ref{secSymmetries}, in general, the kernel integral operators $ K_Q $ do not commute with the Koopman operators $ \tilde U^t$, and thus these families of operators do not share common eigenspaces. Nevertheless, as we establish in this section, under the conditions on kernels stated in Section~\ref{secKernelDelays}, the sequence of operators $ K_Q $ has an asymptotic commutativity property with $ \tilde U^t $ as $Q \to \infty$, allowing the kernel integral operators from VSA to approximate eigenspaces of Koopman operators. 

In order to place our results in context, we begin by noting that an immediate consequence of the bundle construction described in Section~\ref{secBundle} is that if the support of the measure $ \rho $, denoted $ \tilde A \subseteq \Omega $,  is connected as a topological space, then in the limit of no delays, $Q=1$, the image of $ \tilde A $  under the delay coordinate map $ \tilde F_1 $ is a closed interval in $ \mathbb{R}$, and correspondingly the eigenfunctions $ \phi_j $ are pullbacks of orthogonal functions on that interval. In particular, because $ \tilde F_1 $ is equivalent to the vector-valued observation map, in the sense that  $ \vec F( x )( y ) = \tilde F_1( (x, y ) ) $,  the eigenfunctions $ \phi_j$ of the $Q=1$ operator corresponding to nonzero eigenvalues are continuous functions, constant on the level sets of the input signal. Therefore, in this limit, the recovered eigenfunctions will generally have comparable complexity to the input data, and thus be of limited utility for the purpose of decomposing complex signals into simpler patterns. Nevertheless, besides the strict $Q = 1$ limit, the $ \phi_j $ should remain approximately constant on the level sets of the input signal for moderately small values $ Q > 1$, and this property should be useful in a number of applications, such as signal denoising and level set estimation (note that data-driven approximations to $ \phi_j $ become increasingly robust to noise with increasing $Q$ \cite{Giannakis17}). Mathematically, in this small-$Q$ regime VSA has some common aspects with nonlocal averaging techniques in image processing \cite{BuadesEtAl05}.

We now focus on the behavior of VSA in the infinite-delay limit, where the following is found to hold.

\begin{thm}
    \label{thmKoop}
    Under the conditions on the kernels $ k_Q$ stated in Section~\ref{secKernelDelays}, the associated integral operators $K_Q $ converge as $Q \to \infty $ in operator norm, and thus in spectrum, to the integral operator $ K_\infty$ associated with the kernel $ k_\infty $. Moreover, $ K_\infty$ commutes with the Koopman operator  $ \tilde U^t$ for all $ t \in \mathbb{R}$. 
\end{thm}

\begin{proof}
    Since  $ k_Q $ and $k_\infty$ all lie in $H_\Omega \otimes H_\Omega$,  $ K_Q $ and $K_\infty $ are Hilbert-Schmidt integral operators. As a result, the operator norm $ \lVert K_Q - K_\infty \rVert$ is bounded above by $ \lVert k_Q - k_\infty \rVert_{H_\Omega \otimes H_\Omega}$, and the convergence of $ \lVert K_Q - K_\infty \rVert$ to zero follows from the fact that  $ \lim_{Q\to\infty} \lVert k_Q - k_\infty \rVert_{H_\Omega\otimes H_\Omega} =  0 $, as stated in the conditions in Section~\ref{secKernelDelays}. To verify that $K_\infty $ and $ \tilde U^t$ commute, we proceed analogously to the proof of Theorem~\ref{thmSym}, using the shift invariance of $ k_\infty $ in~\eqref{eqKInv} and the fact that $ \tilde \Phi^t = \Phi^t \otimes I_Y $ preserves the measure $ \rho$ to compute
\begin{align*}
   K_\infty f(\omega ) &= \int_\Omega k_\infty( \omega, \omega' ) f( \omega' ) \, d\rho( \omega' ) \\
   &= \int_\Omega k_\infty( \omega, \tilde \Phi^t( \omega' ) ) f( \tilde \Phi^t(  \omega' ) ) \, d\rho( \omega' ) \\
   &= \int_\Omega k_\infty( \tilde \Phi^{-t}( \omega ),  \omega'  ) f( \tilde \Phi^t(  \omega' ) ) \, d\rho( \omega' ) \\
   &= \tilde U^{-t}  K_\infty \tilde U^t f( \omega ),
\end{align*}
where the equalities hold for $ \rho $-a.e.\ $\omega \in \Omega$. Pre-multiplying these expressions by $\tilde U^t $  leads to
\begin{displaymath}
  [ \tilde U^t, K_\infty ] = \tilde U^t K_\infty - K_\infty \tilde U^t = 0,
\end{displaymath} 
as claimed. 
\end{proof}

Theorem~\ref{thmKoop} generalizes the results in  \cite{Giannakis17,DasGiannakis17}, where analogous commutativity properties between Koopman and kernel integral operators were established for scalar-valued observables. By virtue of the commutativity between $ K_\infty $ and $ \tilde U^t $, at large numbers of delays $Q$, VSA decomposes the signal into patterns with a coherent temporal evolution associated with intrinsic frequencies of the dynamical system. In particular, being a compact operator, $ K_\infty$ has finite-dimensional eigenspaces, $ W_l$, corresponding to nonzero eigenvalues, whereas the eigenspaces of $\tilde{U}^t $ are infinite-dimensional, yet are spanned by eigenfunctions with a highly coherent (periodic) time evolution at the corresponding eigenfrequencies $ \alpha_j \in \mathbb{R}$,
\begin{displaymath}
    \tilde U^t \tilde z_j = e^{i\alpha_j t} \tilde z_j,  \quad \tilde z_j \in H_\Omega;
\end{displaymath}
see Appendix~\ref{appKoopEig} for further details. The commutativity between $K_\infty$ and $\tilde{U}^t$ allows us to identify finite-dimensional subspaces $W_l$ of $H_\Omega$ containing distinguished observables which are simultaneous eigenfunctions of $K_\infty$ and $ \tilde U^t $. As shown in Appendix~\ref{appQInf}, these eigenfunctions have the form
\begin{equation}
    \label{eqKoopEigVec}
    \tilde z_{jl} = z_{jl} \otimes \psi_{jl}, \quad j \in \{ 1, \ldots, \dim W_l \},    
\end{equation}
where $z_{jl} $ is an eigenfunction of the Koopman operator $U^t$ on $ H_X $ at eigenfrequency $ \alpha_{jl}$, and $ \psi_{jl}$ a spatial pattern in $H_Y$. Note that here we use a two-index notation, $ z_{jl} $ and $ \alpha_{jl} $, for Koopman eigenvalues and eigenfrequencies, respectively, to indicate the fact that they are associated with the $W_l$ eigenspace of $K_\infty$.  
We therefore deduce from~\eqref{eqKoopEigVec} that in the infinite-delay limit, the spatiotemporal patterns recovered by VSA have a separable, tensor-product structure similar to the conventional decomposition in~\eqref{eqPODDecomp} based on scalar kernel algorithms. It is important to note, however, that unlike~\eqref{eqPODDecomp}, the spatial patterns $ \psi_{jl} $ in~\eqref{eqKoopEigVec} are not necessarily given by linear projections of the observation map onto the corresponding scalar Koopman eigenfunctions $ z_{jl} \in H_X$ (called Koopman modes in the Koopman operator literature \cite{Mezic05}). In effect, taking into account the intrinsic structure of spatiotemporal data as vector-valued observables, allows VSA to recover more general spatial patterns than those associated with linear projections of observed data. 

Another consideration to keep in mind (which applies for many techniques utilizing delay-coordinate maps besides VSA) is that $K_\infty$ can only recover patterns in a subspace  $ \mathcal{D}_\Omega$ of $H_\Omega$ associated with the point spectrum  of the dynamical system generating the data (i.e., the Koopman eigenfrequencies; see Appendix~\ref{appKoopEig}). Dynamical systems of sufficient complexity will exhibit a non-trivial subspace $ \mathcal{D}_\Omega^\perp$ associated with the continuous spectrum, which does not admit a basis associated with Koopman eigenfunctions. One can show via analogous arguments to \cite{DasGiannakis17} that $ \mathcal{D}_\Omega^\perp$ is, in fact, contained in the nullspace of $K_\infty$, which is a potentially infinite-dimensional space not accessible from data. Of course, in practice, one always works with finitely many delays $Q$, which in principle allows recovery of patterns in $ \mathcal{D}_\Omega^\perp$ through eigenfunctions of $K_Q$, and these patterns will not have an asymptotically separable behavior as $Q \to \infty$ analogous to~\eqref{eqKoopEigVec}.

In light of the above considerations, we can therefore conclude that increasing $Q$ from small values will impart changes to the topology of the base space $B_Q$, and in particular the image of the support $ \tilde A $ of $ \rho $ under $\pi_Q$, but also the spectral properties of the operators $K_Q$. On the basis of classical delay-embedding theorems \cite{SauerEtAl91}, one would expect the topology of $\pi_Q(\tilde A)$ to eventually stabilize, in the sense that for every spatial point $y \in Y$ the set $A_y = A \times \{ y \} \subseteq \tilde A $ will map homeomorphically under $ \pi_Q $ for $Q $ greater than a finite number (that is, topologically, $\pi_Q(A_y)$ will be a ``copy'' of $A$). However, apart from special cases, $ K_Q $ will continue changing all the way to the asymptotic limit $Q \to \infty$ where Theorem~\ref{thmKoop} holds. 

Before closing this section, we also note that while VSA does not directly provide estimates of Koopman eigenfrequencies, such estimates could be computed through Galerkin approximation techniques utilizing the eigenspaces of $ K_Q$ at large $Q$ as trial and test spaces, as done in \cite{GiannakisEtAl15,Giannakis17,DasGiannakis17} for scalar-valued Koopman eigenfunctions. A study of such techniques in the context of vector-valued Koopman eigenfunctions (equivalently, eigenfunctions in $H_\Omega$) is beyond the scope of this work, though it is expected that their well-posedness and convergence properties should follow from fairly straightforward modification of the approach in \cite{GiannakisEtAl15,Giannakis17,DasGiannakis17}.

\subsubsection{Infinitely many delays with dynamical symmetries}

As a final asymptotic limit of interest, we consider the limit $ Q \to \infty$ under the assumption that a symmetry group $ G $ acts on $ H_\Omega$ via unitary operators $R^g_\Omega$, as described in Section~\ref{secSymmetries}. In that case, the commutation relations
\begin{displaymath}
  [ R^g_\Omega, \tilde U^t ] = [ R^g_\Omega, K_\infty ] = [ K_\infty, \tilde U^t ] = 0
\end{displaymath}
imply that there exist finite-dimensional subspaces of $ H_\Omega $ spanned by simultaneous eigenfunctions of $ R^g_\Omega $, $ \tilde U^t $, and $ K_\infty $. We know from~\eqref{eqKoopEigVec} that these eigenfunctions, $ \tilde z_{jl}$, are given by a tensor product between a Koopman eigenfunction $ z_{jl} \in H_X$ and a spatial pattern $\psi_{jl} \in H_Y$. It can further be shown (see Appendix~\ref{appSymEig}) that $ z_{jl} $ and $ \psi_{jl}$ are eigenfunctions of the unitary operators $R^g_X $ and $ R^g_Y $ , i.e.,
\begin{displaymath}
    R^g_X z_{jl} = \gamma^g_{X,jl} z_{jl}, \quad R^g_Y \psi_{jl} = \gamma^g_{Y,jl} \psi_{jl}, \quad \lvert \gamma^g_{X,jl} \rvert = \lvert \gamma^{g}_{Y,jl} \rvert = 1,
\end{displaymath}
and moreover the eigenvalues $ \gamma^g_{X,jl}$ and $ \gamma^g_{Y,jl}$ satisfy $ \gamma^g_{X,jl} \gamma^g_{Y,jl} = 1$. In particular, we have $ R^g_\Omega = R^g_X \otimes R^g_Y$, and the quantity $ \gamma^g_{\Omega,jl} = \gamma^g_{X,jl}  \gamma^g_{Y,jl}$ is equal to the eigenvalue of $ R^g_\Omega$ corresponding to $ \tilde z_{jl}$, which is equal to 1 by Theorem~\ref{thmSym}.

In summary, every simultaneous eigenfunction $ \tilde z_{jl}$ of $K_\infty$, $\tilde U^t $, and $ R^g_{\Omega}$ is characterized by three eigenvalues, namely (i) a kernel eigenvalue $ \lambda_l $ associated with $ K_\infty $; (ii) a Koopman eigenfrequency $ \alpha_{jl} $ associated with $ \tilde U^t $; and (iii) a spatial symmetry eigenvalue $ \gamma^g_{Y,jl} $ (which can be thought of as a ``wavenumber'' on $Y$).

\section{\label{secDataDriven}Data-driven approximation}

In this section, we consider the problem of approximating the eigenvalues and eigenfunctions of the kernel integral operators employed in VSA from a finite dataset consisting of time-ordered measurements of the vector-valued observable  $ \vec F$. Specifically, we assume that available to us are measurements $\vec F(x_0), \vec F(x_1), \ldots,\vec F( x_{N-1} ) $  taken along an (unknown) orbit $x_n = \Phi^{n\tau}(x_0)$ of the dynamics at the sampling interval $\tau$, starting from an initial state $x_0 \in X$. We also consider that each scalar field $\vec F(x_n) \in H_Y $ is sampled at a finite collection of points $ y_0, y_1, \ldots, y_{S-1} $ in $Y$. Given such data, and without assuming knowledge of the underlying dynamical flow and/or state space geometry, our goal is to construct a family of operators, whose eigenvalues and eigenfunctions converge, in a suitable sense, to those of $K_Q$, in an asymptotic limit of large data, $N, S \to \infty$. In essence, we seek to address a problem on spectral approximation of kernel integral operators from an unstructured grid of points $( x_n, y_s)$ in $ \Omega $.   

\subsection{ \label{secDataDrivenHilbert}Data-driven Hilbert spaces and kernel integral operators}

An immediate consequence of the fact that the dynamics is unknown is that the invariant measure $\mu$ defining the Hilbert space $H_X = L^2(X,\mu) $ is also unknown (arguably, apart from special cases, $\mu$ would be difficult to explicitly determine even if $\Phi^t$ were known). This means that instead of $H_X$ we only have access to an $N$-dimensional Hilbert space $H_{X,N} = L^2(X,\mu_N) $ associated with the sampling measure $ \mu_N = \sum_{n=0}^{N-1} \delta_{x_n}/N$ on the trajectory $X_N = \{ x_0, \ldots,x_{N-1} \}$, where $ \delta_{x_n} $ is the Dirac probability measure supported at $ x_n \in X $. This space consists of equivalence classes of functions on $X$ having common values on the finite set  $X_N  \subset X $, and is equipped with the inner product $ \langle f, g \rangle_{H_{X,N}} = \sum_{n=0}^{N-1} f^*(x_n) g(x_n)/ N$. Because every such equivalence class $f$  is uniquely characterized by $N$ complex numbers, $ f(x_0), \ldots, f(x_{N-1})$,  corresponding to the values of one of its representatives  on $X_N $, $H_{X,N}$ is isomorphic to $ \mathbb{C}^N$,  equipped with a normalized Euclidean inner product. Thus, we can represent every $ f \in H_{X,N}$ by an $N$-dimensional column vector $ \underline f = ( f(x_0), \ldots, f(x_{N-1} ) )^\top \in \mathbb{C}^N$, and every linear operator $ A : H_{X,N} \mapsto H_{X,N} $ by an $N\times N $ matrix $ \boldsymbol A $ such that $ \boldsymbol A \underline f $ is equal to the column-vector representation of $ A f $. In particular, associated with every scalar kernel $ k : X \times X \mapsto \mathbb{R} $ is a kernel integral operator $ K_N : H_{X,N} \mapsto H_{X,N} $, acting on $f \in H_{X,N}$ according to the formula (cf.\ \eqref{eqKScalar}) 
\begin{equation}
    \label{eqKN}
    K_N f(x_m ) = \int_X k(x_m,x_n) \, d\mu_N(x_n) = \frac{1}{N} \sum_{n=0}^{N-1} k(x_m,x_n) f( x_n ).
\end{equation}
This operator is represented by an $N\times N$ kernel matrix $ \boldsymbol K = [ k(x_m,x_n ) / N ] $. 

In the setting of spatiotemporal data analysis, one has to also take into account the finite sampling of the spatial domain, replacing  $H_Y =L^2(Y,\nu)$ by the $S$-dimensional Hilbert space  $H_{Y,S} = L^2(Y,\nu_S)$ associated with a discrete measure $ \nu_S = \sum_{s=0}^{S-1} w_{s,S} \delta_{y_s} $. Here, the $w_{s,S} $ are positive quadrature weights such that given any continuous function $f : Y \mapsto \mathbb{C}$, the quantity $ \sum_{s=0}^{S-1} w_{s,S} f(y_s) $ approximates $ \int_Y f \, d\nu$. For instance, if if $ \nu$ is a probability measure, and the sampling points $ y_s $ are equidistributed with respect to $ \nu $, a natural choice is uniform weights, $ w_{s,S} = 1/ S $. The space $H_{Y,S}$ is constructed analogously to $H_{X,N}$, and similarly we replace $H_\Omega = L^2(\Omega,\rho)$ by the $NS$-dimensional Hilbert space $H_{\Omega,NS} = L^2 (\Omega, \rho_{NS}) $, where $ \rho_{NS} = \mu_N \times \nu_S = \sum_{n=0}^{N-1} \sum_{s=0}^{S-1} w_{s,S} \delta_{\omega_{ns}} /N $ and  $ \omega_{ns} = ( x_n, y_s)$. Given a kernel $ k_Q : \Omega \times \Omega \mapsto \mathbb{R} $ satisfying the conditions in Section~\ref{secKernelDelays}, there is an associated integral operator $K_{Q,NS} : H_{\Omega,NS} \mapsto H_{\Omega,NS} $, defined analogously to~\eqref{eqKN} by
\begin{equation}
    \label{eqKNS}
    K_{Q,NS} f( \omega_{mr} ) = \int_{\Omega} k_Q( \omega_{mr}, \omega_{ns} ) \, d\rho_{NS}(\omega_{ns}),
\end{equation}
and represented by the $(NS)\times (NS) $ matrix $ \boldsymbol K = [ k_{Q}( \omega_{mr}, \omega_{ns} ) / N ]$. Solving the eigenvalue problem for $K_{Q,NS} $ (which is equivalent to the matrix eigenvalue problem for $ \boldsymbol K$) leads to eigenvalues $  \lambda_{j,NS} \in \mathbb{R} $  and eigenfunctions $ \phi_{j,NS} \in H_{\Omega,NS} $. We consider $ \lambda_{j,NS} $ and $ \phi_{j,NS} $  as data-driven approximations to their eigenvalues and eigenfunctions of $K_Q $ from~\eqref{eqKOmega}, respectively. The convergence properties of this approximation will be made precise in Section~\ref{secSpecConv}. 

A similar data-driven approximation can be performed for operators based on the Markov kernels $ p_Q $ from  Section~\ref{secMarkov}, which is our preferred class of kernels for VSA. However, in this case the kernels $ p_{Q,NS} : \Omega \times \Omega \mapsto \mathbb{R}$ associated with the approximating operators $P_{Q,NS} $ on $H_{\Omega,NS}$ are Markov-normalized with respect to the measure $ \rho_{NS}$, i.e., $ \int_\Omega p_{Q,NS}(\omega, \cdot) \, d\rho_{NS} = 1$, so they acquire a dependence on $N$ and $S$. Further details on this construction and its convergence properties can be found in Appendix~\ref{appDataDriven}. 

\subsection{\label{secSpecConv}Spectral convergence}

For a spectrally-consistent data-driven approximation scheme, we would like to able to establish that,   as $N$ and $S$ increase, the sequence of eigenvalues $  \lambda_{j,NS}$ of $K_{NS}$ converges to eigenvalue  $ \lambda_j $ of $ K $, and for an eigenfunction $ \phi_j$ of $K$ corresponding to $\lambda_j $ there exists a sequence of eigenfunctions $  \phi_{j,NS} $ of $ K_{NS}$ converging to it. While convergence of eigenvalues can be unambiguously understood in terms of convergence of real numbers, in the setting of interest here a suitable notion of convergence of eigenfunctions (or, more generally, eigenspaces) is not obvious, since $  \phi_{j,NS}$ and $ \phi_j $ lie in fundamentally different spaces. That is, there is no natural way of mapping equivalence classes of functions with respect to $ \rho_{NS}$ (i.e., elements of $H_{\Omega,NS}$) to equivalence classes of functions with respect to $ \rho$ (i.e., elements of $H_{\Omega}$), allowing one, e.g.,  to establish convergence of eigenfunctions in $H_\Omega$ norm. This issue is further complicated by the fact, that in many cases of interest, the support $A$ of the invariant measure $\mu$ is a non-smooth subset of $X$ of zero Lebesgue measure (e.g., a fractal attractor), and the sampled states $x_n $ do not lie exactly on $A$ (as that would require starting states $x_0 $ drawn from a measure zero subset of $X$, which is not feasible experimentally). In fact, the issues outlined above are common to many other data-driven techniques for analysis of dynamical systems besides VSA (e.g., POD and DMD), yet to our knowledge have not received sufficient attention in the literature. 

    Here, following \cite{DasGiannakis17}, we take advantage of the fact that, by the assumed continuity of VSA kernels, every kernel integral operator $K_Q : H_{\Omega} \mapsto H_{\Omega} $ from Section~\ref{secKernelDelays} can be also be viewed as an integral operator on the space $C(\mathcal{V})$ of continuous functions on any compact subset $ \mathcal{V} \subset \Omega $ containing the support of $ \rho $. This integral operator, denoted by $ \tilde K_Q : C(\mathcal{V}) \mapsto C(\mathcal{V})$, acts on continuous functions through the same integral formula as~\eqref{eqKOmega}, although the domains of $K_Q $ and $ \tilde K_Q$ are different. It is straightforward to verify that every eigenfunction $ \phi_j \in H_\Omega $ of $K_Q$ at nonzero eigenvalue $ \lambda_j $ has a unique continuous representative $ \tilde \phi_j \in C(\mathcal{V}) $, given by
    \begin{equation}
        \label{eqTildePhiK}
        \tilde \phi_j(\omega) = \frac{1}{\lambda_j} \int_{\Omega} k_Q( \omega, \omega' ) \phi_j( \omega' ) \, d\rho(\omega'), 
    \end{equation}
    and $ \tilde \phi_j $ is an eigenfunction of $ \tilde K_Q $ at the same eigenvalue $ \lambda_j $. Assuming further that $ \mathcal{V}$ also contains the supports of the measures $ \rho_{NS}$ for all $N,S \geq 1 $, we can define $ \tilde K_{Q,NS} : C(\mathcal{V}) \mapsto C(\mathcal{V})$ analogously to~\eqref{eqKNS}. Then, every eigenfunction $ \phi_{j,NS} \in H_{\Omega,NS} $ of $ K_{Q,NS}$ at nonzero corresponding eigenvalue $ \lambda_{j,NS}$ has a unique continuous representative $ \tilde \phi_{j,NS}$, with
    \begin{equation}
        \label{eqTildePhiKNS}
        \tilde \phi_{j,NS}(\omega) = \frac{1}{\lambda_j} \int_{\Omega} k_{Q}( \omega, \omega' ) \phi_{j,NS}( \omega' ) \, d\rho_{NS}(\omega'), 
    \end{equation}
    which is an eigenfunction of $ \tilde K_{Q,NS}$ at the same eigenvalue $ \lambda_{j,NS}$. 

    As is well known, the space $C(\mathcal{V})$ equipped with the uniform norm $ \lVert f \rVert_{C(\mathcal{V})} = \max_{\omega\in\mathcal{V}} \lvert f(\omega) \rvert$ becomes a Banach space, and it can further be shown that $\tilde K_{Q,NS}$ and $ \tilde K_Q$ are compact operators on this space. In other words, $C(\mathcal{V})$ can be used as a universal space to establish spectral convergence of $ \tilde K_{Q,NS}$ to $ \tilde K_Q$, using approximation techniques for compact operators on Banach spaces \cite{Chatelin11}. In \cite{VonLuxburgEtAl08}, Von Luxburg et al.\ use this approximation framework to establish convergence results for spectral clustering techniques, and their approach can naturally be adapted to show that, under natural assumptions, $ \tilde K_{Q,NS}$ indeed converges in spectrum to $ \tilde K_Q$. 

    In Appendix~\ref{appDataDriven}, we prove the following result:

\begin{thm}
    \label{thmDataDriven}
    Suppose that $\mathcal{V} \subseteq \Omega $ is a compact set containing the supports of $ \rho $ and the family of measures $ \rho_{NS} $, and assume that  $ \rho_{NS}$ converges weakly to $ \rho$, in the sense that 
    \begin{equation}
        \label{eqPhysicalM}
        \lim_{N,S\to\infty} \int_\Omega f \, d\rho_{NS} = \int_\Omega f \, d\rho, \quad \forall f \in C(\Omega).
    \end{equation}
    Then, for every nonzero eigenvalue $ \lambda_j $ of $ K_Q $, including multiplicities, there exist positive integers $ N_0, S_0   $ such that the eigenvalues $ \lambda_{j,NS} $ of $ K_{Q,NS} $ with $ N \geq N_0 $ and $ S \geq S_0$  converge, as $ N,S \to \infty $, to $ \lambda_j $. Moreover, for every eigenfunction $ \phi_j \in H_\Omega $ of $ K $ corresponding to $\lambda_j$, there exist eigenfunctions $ \phi_{j,NS} $ of $ K_{Q,NS} $ corresponding to $ \lambda_{j,NS} $,  whose continuous representatives $ \tilde \phi_{j,NS} $ from~\eqref{eqTildePhiKNS} converge uniformly on $ \mathcal{ V } $ to $ \tilde \phi_j $ from~\eqref{eqTildePhiK}. Moreover, analogous results hold for the eigenvalues and eigenfunctions of the Markov operators $P_{Q,NS} $ and $P_Q$.
\end{thm}

A natural setting where the conditions stated in Theorem~\ref{thmDataDriven} are satisfied are dynamical systems with compact absorbing sets and associated physical measures. Specifically, for such systems we shall assume that there exists a Lebesgue measurable subset $\mathcal{U}$ of the state space manifold $X$, such that (i) $\mathcal{U}$ is forward-invariant, i.e., $ \Phi^t(\mathcal{U}) \subseteq \mathcal{U}$ for all $t\geq 1$; (ii) the topological closure $ \overline{\mathcal{U}}$ is a compact set containing the support of $\mu$; (iii) $\mathcal{U}$ has positive Lebesgue measure in $X$; and (iv) for any starting state $x_0 \in \mathcal{U}$, the corresponding sampling measures $\mu_N$ converge weakly to $\mu$, i.e., $ \lim_{N\to\infty} \int_X f \, d\mu_N = \int_X f \, d\mu$ for all $f \in C(X)$. Invariant measures exhibiting Properties (iii) and (iv) are known as physical measures \cite{Young02}; in such cases the set $\mathcal{U}$ is called a basin of $\mu$. Clearly, Properties (i)--(iv) are satisfied if $ \Phi^t : X \mapsto X $ is flow on a compact manifold with an ergodic invariant measure supported on the whole of $X$, but are also satisfied in more general settings, such as certain dissipative flows on noncompact manifolds (e.g., the Lorenz 63 system on $X=\mathbb{R}^3$ \cite{Lorenz63}). Assuming further that the measures $ \nu_S$ associated with the sampling points $y_0,\ldots, y_{S-1} $ and the corresponding quadrature weights $ w_{0,S}, \ldots, w_{S-1,S} $ on the spatial domain $Y$ converge weakly to $\nu$, i.e., $ \lim_{S\to\infty} \int_Y g \, d\nu_S = \int_Y g \, d\nu $ for every  $ g \in C(Y)$, the conditions in Theorem~\ref{thmConv} are met with $ \mathcal{V} = \overline{\mathcal{U}} \times Y $, and the measures $ \rho_{NS}$ constructed as described in Section~\ref{secDataDrivenHilbert} for any starting state $x_0 \in \mathcal{U}$. Under these conditions, the data-driven spatiotemporal patterns $ \phi_{j,NS}$ recovered by VSA converge for an experimentally accessible set of initial states in $X$. 

\section{\label{secApplications}Application to the Kuramoto-Sivashinsky Model}

\subsection{\label{secKSOverview}Overview of the Kuramoto Sivashinsky model}
The KS model, originally introduced as a model for wave propagation in a dissipative medium \cite{KuramotoTsuzuki76}, or laminar flame propagation \cite{Sivashinsky77}, is one of the most widely studied dissipative PDEs displaying spatiotemporal chaos.  On a one-dimensional spatial domain $ Y = [ 0, L ], L \geq 0 $, the governing evolution equation for the real-valued scalar field $ u( t, \cdot ) : Y \mapsto \mathbb{  R} $, $ t \geq 0 $, is given by
\begin{equation}
  \label{eqKS}
  \dot u =  - u \nabla u + \Delta u - \Delta^2 u,  
\end{equation}
where $ \nabla $ and $ \Delta = - \nabla^2 $ are the derivative and (positive definite) Laplace operators on $ Y$, respectively. In what follows, we always work with periodic boundary conditions, $ u( t, 0 ) = u(t, L) $, $ \nabla u(t, 0 ) = \nabla u( t, L ) $, \ldots, for all $t \geq 0 $. 

The domain size parameter $ L $ controls the dynamical complexity of the system. At small values of this parameter, the trivial solution $u=0$ is globally asymptotically stable, but as $L $ increases, the system undergoes a sequence of bifurcations, marked by the appearance of steady spatially periodic modes (fixed points), then traveling waves (periodic orbits), and progressively more complicated solutions leading to chaotic behavior for $ L \gtrsim 4 \times 2 \pi $ \cite{GreeneKim88,ArmbrusterEtAl89,KevrekidisEtAl90,CvitanovicEtAl09,TakeuchiEtAl11}.

A fundamental property of the KS system is that it possesses a global compact attractor, embedded within a finite-dimensional inertial manifold of class $ C^r $, $ r \geq 1 $ \cite{FoiasEtAl86,FoiasEtAl88,ConstantinEtAl89,JollyEtAl90,ChowEtAl92,Robinson94}. That is, there exists a $C^r$ submanifold $ \mathcal{ X } $ of the Hilbert space $ H_Y = L^2( Y, \nu ) $ with $ \nu $ set to the Lebesgue measure, which is invariant under the dynamics, and to which the solutions $ u( t, \cdot ) $ are exponentially attracted. This means that after the decay of initial transients, the effective degrees of freedom of the KS system, bounded above by the dimension of $ \mathcal{ X } $,  is finite. Dimension estimates of inertial manifolds \cite{Robinson94,JollyEtAl00} and attractors \cite{TajimaGreenside02} of the KS system as a function of  $ L $ indicate that the system exhibits extensive chaos, i.e., unbounded growth of the attractor dimension with $ L $.  As is well known, analogous results to those outlined above are not available for many other important models of complex spatiotemporal dynamics such as the Navier-Stokes equations.  

For our purposes, the availability of strong theoretical results and rich spatiotemporal dynamics makes the KS model particularly well-suited  to test the VSA framework. In our notation, an inertial manifold $ \mathcal{ X } $ of the KS system will act as the state space manifold $ X $, which is embedded in this case in $ H_Y $.  Moreover, the compact invariant set $ A $ will be a subset of the global attractor supporting an ergodic probability measure, $ \mu $. On $ X $, the dynamics is described by a $ C^r $ flow map $ \Phi^t : X \mapsto X $, $ t \in \mathbb{ R } $,  as in Section~\ref{secPrelim}. In particular, for every initial condition $ x_0 \in X $, the orbit $ t \mapsto x( t ) = \Phi^t( x_0 ) $ with $ t \geq 0 $ is the unique solution to~\eqref{eqKS} with initial condition $ x_0 $. While in practice the initial data will likely not lie on $ X $, the exponential tracking property of the dynamics ensures that for any admissible initial condition $ u \in H_Y $ there exists a trajectory $ x( t ) $ on $ X$ to which the evolution starting from $ u $ converges exponentially fast.  

As stated in Section~\ref{secSpecConv}, for data-driven approximation purposes, we will formally assume that  the measure $ \mu$ is physical. While, to our knowledge, there are no results in the literature addressing the existence of physical measures (with appropriate modifications to account for the infinite state space dimension) specifically for the KS system, recent results \cite{LuEtAl13,LianEtAl16} on infinite-dimensional dynamical systems that include the class of dissipative systems in which the KS system belongs indicate that analogs of the assumptions made in Section~\ref{secSpecConv} should hold.  

Another important feature of the KS system is that it admits nontrivial symmetry group actions on the spatial domain $ Y $, which have played a central role in bifurcation studies of this system \cite{GreeneKim88,ArmbrusterEtAl89,KevrekidisEtAl90,CvitanovicEtAl09}. In particular, it is a direct consequence of the structure of the governing equation~\eqref{eqKS} and the periodic boundary conditions that if $ u(x,t) $ is a solution, then so are $ u( x + \alpha, t )$ and $ u( -x, t ) $, where $ \alpha \in \mathbb{ R } $. As discussed in Section~\ref{secSymmetries}, this implies that the dynamics on the inertial manifold is equivariant under the actions induced by the orthogonal group $ O(2) $ and the reflection group on the circle. In particular, under the assumption that the $O(2)$ action preserves $\mu$,  the theoretical spatial patterns recovered by POD and comparable eigendecomposition techniques would be linear combinations of finitely many Fourier modes \cite{HolmesEtAl96}, which are arguably non-representative of the complex spatiotemporal patterns generated by the KS system. We emphasize that the existence of symmetries does not necessarily imply that they are inherited by data-driven operators for extracting spatial and temporal patterns constructed from a single orbit of the dynamics, since, e.g., the ergodic measure sampled by that orbit may evolve singularly under the symmetry group action.  While studies have determined that this type of symmetry breaking indeed occurs at certain dynamical regimes of the KS system \cite{AubryEtAl93}, the presence of symmetries still dominates the leading spatial patterns recovered by POD and comparable eigendecomposition techniques utilizing scalar-valued kernels. 

\subsection{\label{secKSResults}Analysis datasets}

In what follows, we present applications of VSA to data generated by the KS model at the chaotic regimes $ L = 22 $ and $ L = 94 $; these ``standard'' regimes have been investigated extensively in the literature (e.g., \cite{CvitanovicEtAl09,TakeuchiEtAl11}). We have integrated the model using the publicly available Matlab code accompanying Ref.~\cite{CvitanovicEtAl16}. This code is based on a Fourier pseudospectral discretization, and utilizes a fourth-order exponential time-differencing Runge-Kutta integrator appropriate for stiff problems. Throughout, we use 65 Fourier modes (which is equivalent to a uniform grid on $ Y $ with $ S = 65 $ gridpoints and uniform quadrature weights, $ w_{s,S} = 1 /S$), and a timestep of $ \tau = 0.25 $ natural time units. 

Each of the experiments described below starts from initial conditions given by setting the first four Fourier coefficients to 0.6 and the remaining 61 to zero. Before collecting data for analysis, we let the system equilibrate near its attractor for a time interval of 2500 natural time units.  We compute spatiotemporal patterns using the eigenfunctions $ \phi_{j,NS} $ of the data-driven Markov operator $ P_{Q,NS} $ as described in Section~\ref{secDataDriven}. For simplicity, for the rest of this section we will drop the $N$ and $S$ subscripts from our notation for $ \phi_{j,NS} $.  

In one of our $ L = 22 $ experiments, we also compare our results with spatiotemporal patterns computed via POD/PCA and NLSA (see Section~\ref{secBackground}).  The PCA and NLSA methods are applied to the same KS data as VSA, and in the case of NLSA we use the same number of delays. The POD patterns are computed via~\eqref{eqPODDecomp}, whereas those from NLSA are obtained via a procedure originally introduced in the context of SSA \cite{GhilEtAl02}. This procedure involves first reconstructing in delay-coordinate space through~\eqref{eqPODDecomp} applied to the observation map $\tilde F_Q $ from~\eqref{eqFQ}, and then projecting down to physical data space by averaging over consecutive delay windows; see \cite{GhilEtAl02,GiannakisMajda12a} for additional details. Empirically, this reconstruction approach is known to be more adept at capturing propagating signals than direct reconstruction of the observation map $ F $ via~\eqref{eqPODDecomp}, though in the KS experiments discussed below the results from the two-step NLSA/SSA reconstruction and direct reconstruction are very similar. 

\subsection{Results and discussion}

We begin by presenting VSA results obtained from dataset of $ N = 1000 $ samples taken at timestep $\tau = 0.25$ natural time units, using a small number of delays, $ Q = 15 $.  According to Section~\ref{secInfDelay}, at this small $ Q $ value VSA is expected to yield eigenfunctions $\phi_{j}$, which are approximately constant on the level sets of the input signal, and, with increasing $ j $, capture smaller-scale variations in the directions transverse to the level sets. As is evident in Fig.~\ref{figKS_L22_Q15}, the leading three nonconstant eigenfunctions, $ \phi_{1} $, $ \phi_{2} $, and $ \phi_{3} $, indeed display this behavior, featuring wavenumbers 2, 3, and 4, respectively, in the directions transverse to the level sets. This behavior continues for eigenfunctions $ \phi_{j} $ with higher $ j $. 

To assess the efficacy of these patterns in reconstructing the input signal, we compute their fractional explained variances 
\begin{equation}
    \label{eqFVar}
  \frac{ \lvert \langle F, \phi_{j} \rangle_{H_{\Omega,NS}}\rvert^2 }{ \lVert \phi_{j} \rVert_{H_{\Omega,NS}}^2 \lVert F \rVert_{H_{\Omega,NS}}^2 }, 
\end{equation}
where $\lVert \phi_{j} \rVert_{H_{\Omega_,NS}} = 1$ since we use normalized eigenfunctions. For the  $ \phi_{1} $, $ \phi_{2} $, and $ \phi_{3} $ eigenfunctions in Fig.~\ref{figKS_L22_Q15}, these quantities are 0.91, $5.2\times10^{-4}$, and 0.016, respectively, which demonstrates that even the one-term ($l=1$) reconstruction via~\eqref{eqVSADecomp} captures most of the signal variance.

\begin{figure}
    \centering\includegraphics[width=.5\linewidth]{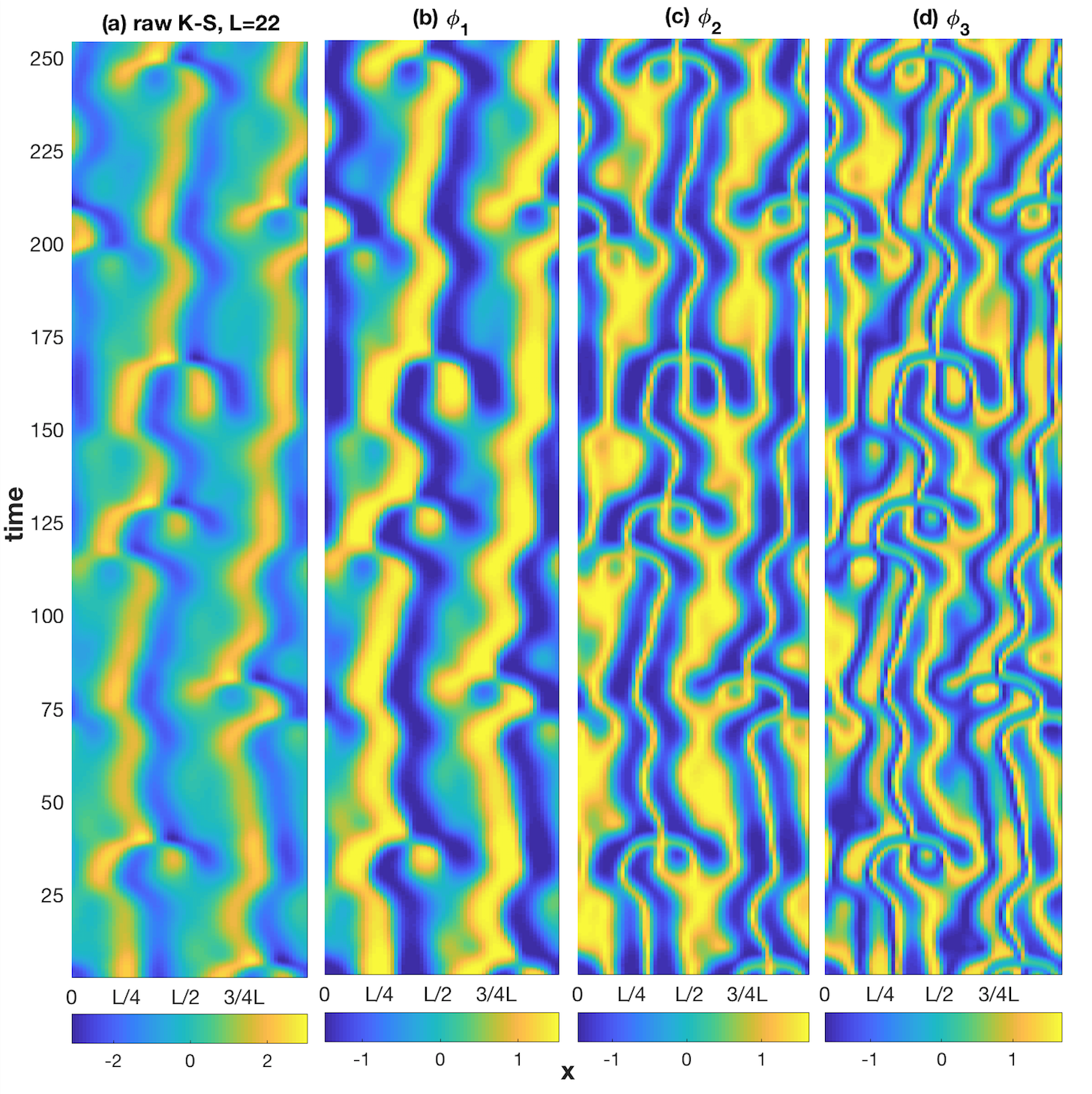}
    \caption{\label{figKS_L22_Q15}Raw data (a) and representative VSA spatiotemporal patterns $ \phi_{j} $ (b--d) for the KS model with $ L = 22 $ using $ Q = 15 $ delays. Notice that the patterns are approximately constant on the level sets of the raw data.}
\end{figure}

Next, we consider longer datasets with $ N = \text{10,000} $ samples (2500 natural time units), at $ L = 22 $ and 94, analyzed using $ Q = 500 $ delays. The raw data and representative VSA eigenfunctions from these analyses, as well as NLSA results for $L=22$, are displayed in Figures~\ref{figKS_L22_Q500} and~\ref{figKS_L94_Q500}, respectively. Figure~\ref{figKS_L22_Q500_detail} highlights a portions of the raw data and VSA eigenfunctions for $L=22$ over an interval spanning 1000 time units. Figure~\ref{figKS_L22_PCA} shows the $ L = 22 $ raw data and the leading five (in ordered of explained variance) spatiotemporal patterns recovered from this dataset by POD. As is customary, we order the recovered POD patterns in order of decreasing explained variance of the input data. In the case of NLSA, the patterns are ordered in decreasing order of the corresponding eigenvalue of the Markov operator associated with the NLSA kernel (analogous to the $ P_Q $ operator in VSA, but acting on the $ H_X $ space of scalar-valued observables). Note that since the vector-valued eigenfunctions from VSA are directly interpretable as spatiotemporal patterns (see Section~\ref{secVSA}), and the VSA decomposition from~\eqref{eqVSADecomp} is given by linear combinations of eigenfunctions with scalar-valued coefficients $c_j $, comparing VSA eigenfunctions with PCA and NLSA spatiotemporal patterns (which are formed by products of scalar-valued eigenfunctions of the corresponding kernel integral operators with spatial patterns) is meaningful, but since our depicted VSA patterns do not include multiplication by $ c_j $, these comparisons are to be made only up to scale.

At large $ Q $, we expect the eigenfunctions from VSA to lie approximately in finite-dimensional subspaces of the Hilbert space $ H $ of vector-valued observables associated with the point spectrum of the Koopman operator, thus acquiring timescale separation. This is clearly the case in the $L=22 $ eigenfunctions in Figs.~\ref{figKS_L22_Q500} and~\ref{figKS_L22_Q500_detail}, where $ \phi_{1} $ is seen to capture the evolution of the wavenumber $L/2$ structures, whereas $ \phi_{10} $ and $ \phi_{15} $ recover smaller-scale traveling waves embedded within the large-scale structures with a general direction of propagation either to the right ($\phi_{10}$) or left ($\phi_{15}$). The fractional explained variances associated with eigenfunctions $\phi_{1}$, $\phi_{10}$, and $\phi_{15}$ are 0.23, 0.047, and 0.017, respectively. As expected, these values are smaller than the 0.91 value due to eigenfunction $\phi_{1}$ for $Q=15$, but are still fairly high despite the intermittent nature of the input signal. Ranked with respect to fractional explained variance, $\phi_{1}$, $\phi_{10}$, and $\phi_{15}$ are the first, fourth, and fifth among the $Q=200$ VSA eigenfunctions. 

\begin{figure}
    \includegraphics[width=\linewidth]{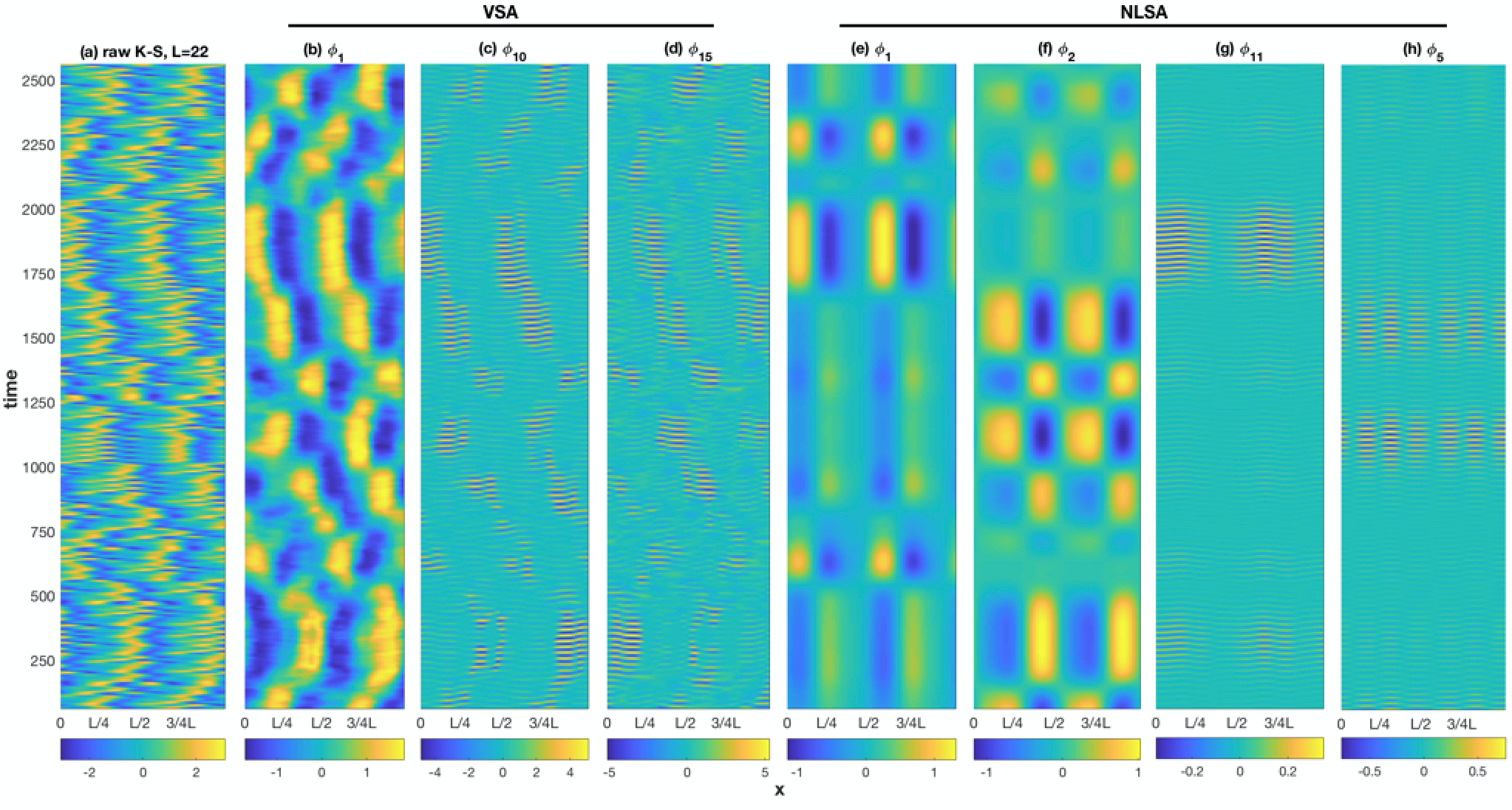}
    \caption{\label{figKS_L22_Q500} Raw data (a) and representative spatiotemporal patterns obtained through VSA (b--d) and NLSA (e--h) from the KS model with $ L = 22 $ using $ Q = 500 $ delays. The VSA pattern in (b) captures  an $O(2)$ family of unstable equilibria through an individual eigenfunction. Notice the manifestly non-separable character of this pattern with respect to the spatial and temporal coordinates. The patterns in (c, d) capture smaller-scale waves embedded within the structures in (b). The NLSA patterns in (e--h) exhibit a low-rank, separable behavior in space and time, and while they appear to capture the characteristic timescales of the large- (e, f) and small-scale waves (g, h), they are not representative of the intermittent character of the signal in space.}
\end{figure}

\begin{figure}
    \centering\includegraphics[width=\linewidth]{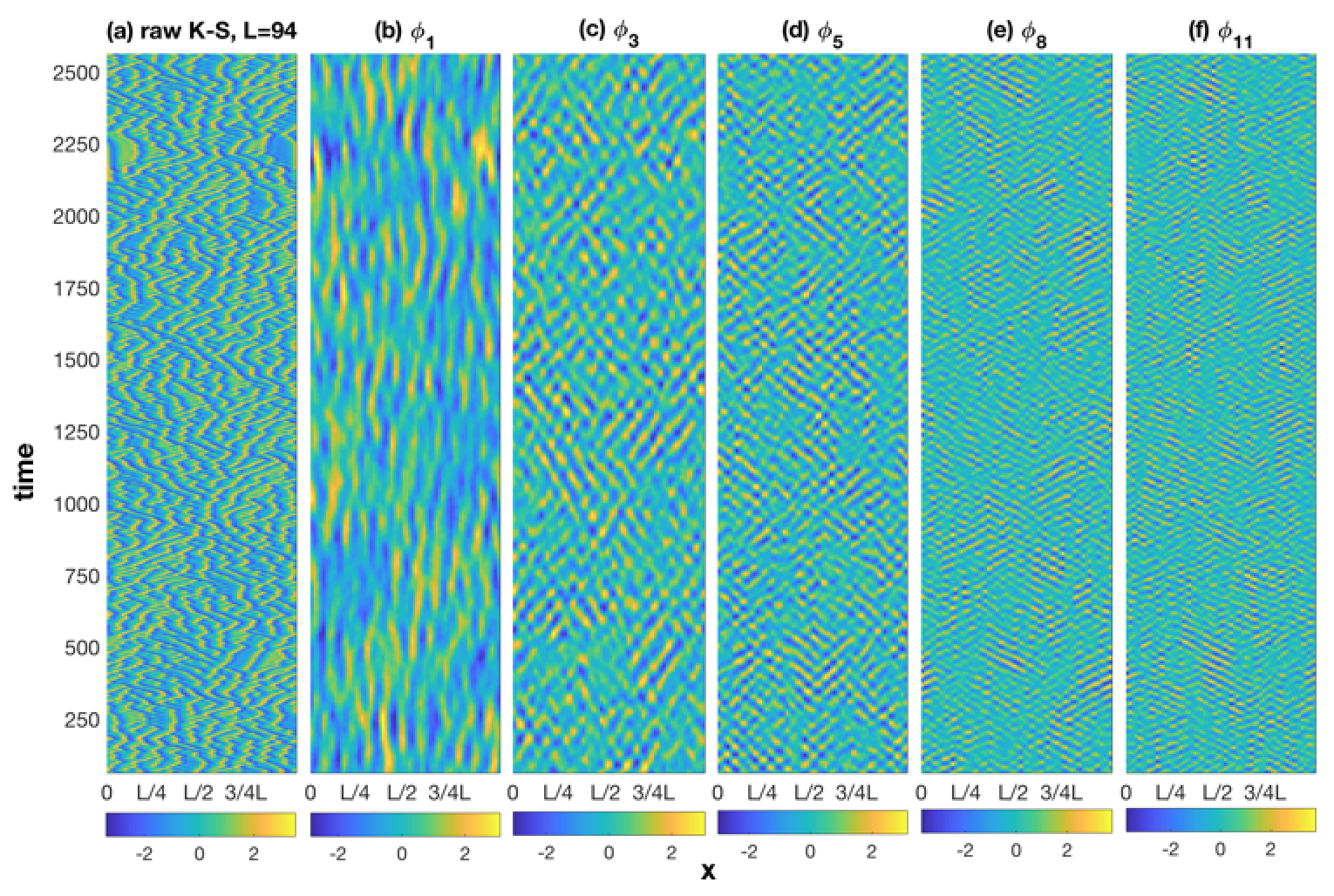}
    \caption{\label{figKS_L94_Q500}As in Fig.~\ref{figKS_L22_Q500}(a--d), but for the KS system with $ L= 94 $. (b) Vector-valued eigenfunctions capturing an $O(2)$ family of unstable equilibria. (c--f) Traveling-wave patterns.}
\end{figure}

\begin{figure}
    \centering\includegraphics[width=.5\linewidth]{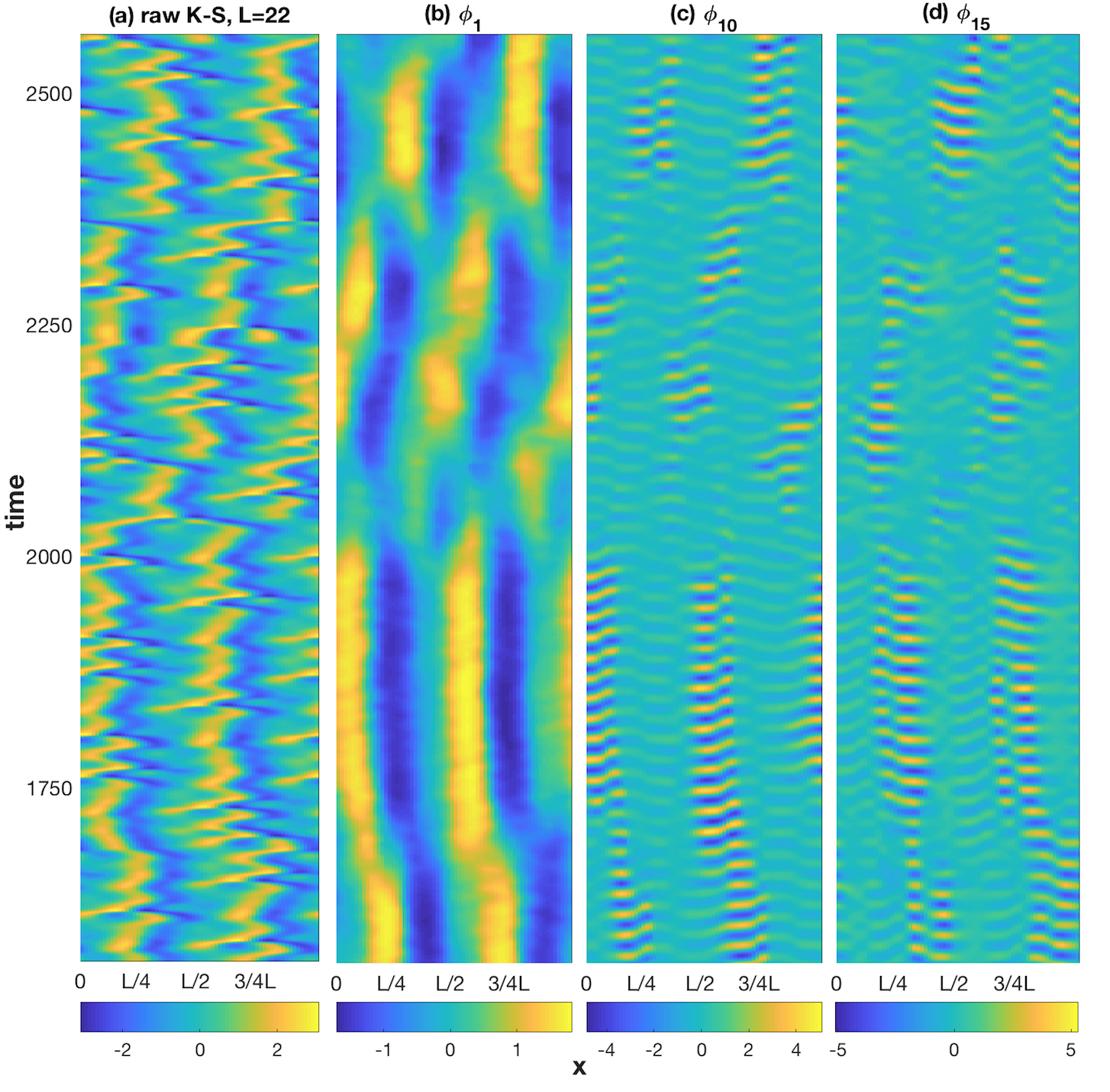}
    \caption{\label{figKS_L22_Q500_detail}As in Fig.~\ref{figKS_L22_Q500}(a--d), but for the time interval $ [ 1625, 2625] $, highlighting the features of the small-scale waves in (c, d).}
\end{figure}

\begin{figure}
    \centering\includegraphics[width=\linewidth]{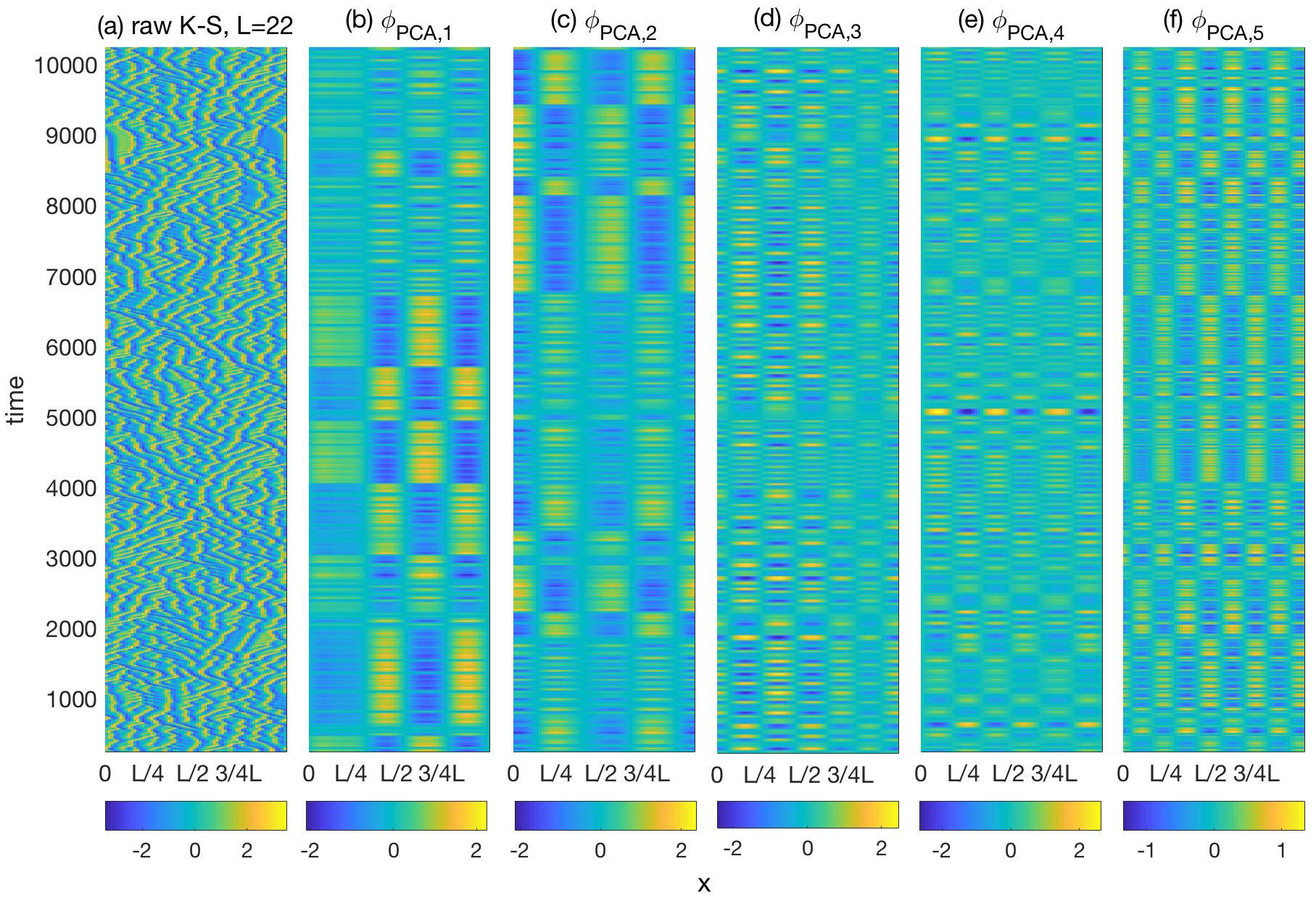}
    \caption{\label{figKS_L22_PCA}As in Fig.~\ref{figKS_L22_Q500}, but for spatiotemporal patterns recovered by POD/PCA.}
\end{figure}

In contrast, while the patterns from NLSA successfully separate the slow and fast timescales in the input signal (as expected theoretically at large $ Q $ \cite{Giannakis17,DasGiannakis17}), they are significantly less efficient in capturing its salient spatial features.  Consider, for example, the leading two NLSA patterns shown in Fig.~\ref{figKS_L22_Q500}(e,f). These patterns are clearly associated with the $O(2)$ family of wavenumber $L/2$ structures in the raw data, but because they have a low rank, they are unable to represent the intermittent spatial translations of these patterns produced by chaotic dynamics in this regime. Their fractional explained variances are 0.13 and 0.15, respectively.  Qualitatively, it appears that the NLSA patterns in Fig.~\ref{figKS_L22_Q500}(e, f) isolate periods during which the wavenumber $L/2$ structures are quasistationary, and translated relative to each other by $L/4$. In other words, it appears that NLSA captures the unstable equilibria that the system visits in the analysis time period through individual patterns, but does not adequately represent the transitory behavior associated with heteroclinic orbits connecting this family of equilibria. Moreover, due to the presence of the continuous $O(2)$ symmetry, a complete description of the spatiotemporal signal associated with the wavenumber $L/2$ structures would require several modes. In contrast, VSA effectively captures this dynamics through a small set of leading eigenfunctions. 

As can be seen in Fig.~\ref{figKS_L22_PCA}, POD would also require several modes to capture the wavenumber $ L/2$ unstable equilibria, but in this case the recovered patterns also exhibit an appreciable amount of mixing of the slow timescale characteristic of this family with faster timescales. Modulo this high-frequency mixing, the first (second) POD pattern appears to resemble the first (second) NLSA pattern. The fractional explained variance of the leading two POD patterns, amounting to 0.23 and 0.22, respectively, is higher than the corresponding variances from NLSA, but this is not too surprising given their additional frequency content. 

To summarize, these results demonstrate that NLSA improves upon PCA in that it achieves timescale separation through the use of delay-coordinate maps, and VSA further improves upon NLSA in that it quotients out the $O(2) $ symmetry of the system, allowing efficient representation of intermittent space-time signals associated with heteroclinic dynamics in the presence of this symmetry. In separate calculations, we have verified that the $L=22$ VSA patterns are robust under corruption of the data by i.i.d.\ Gaussian noise of variance up to 40\% of the raw signal variance. 
    
Turning now to the $L=94$ experiments, it is evident from Fig.~\ref{figKS_L94_Q500}(b) that the dynamical complexity in this regime is markedly higher than for $ L = 22 $, as multiple traveling and quasistationary waves can now be accommodated in the domain, resulting in a spatiotemporal signal with high intermittency in both space and time. Despite this complexity, the recovered eigenfunctions (Fig.~\ref{figKS_L94_Q500}(b--f)) decompose the signal into a pattern $\phi_{1}$ that captures the evolution of unstable fixed points and the heteroclinic connections between them, and other patterns, $ \phi_{3} $, $ \phi_{5} $, $\phi_{8}$, and  $\phi_{11} $,  dominated by traveling waves. The fractional explained variances associated with these patterns are $6.0 \times 10^{-3}$ ($\phi_{1}$), 0.020 ($\phi_{3}$), 0.040 ($\phi_{5}$), 0.067 ($\phi_{8}$), and 0.066 ($\phi_{11}$); that is, in this regime the traveling wave patterns are dominant in terms of explained variance. In general, the variance explained by individual eigenfunctions at $L=94$ is smaller than those identified for $L=22$, consistent with the higher dynamical complexity of the former regime. It is worthwhile noting that $L=94$ eigenfunction $\phi_{1}$ bears some qualitative similarities with the covariant Lyapunov vector (CLV) patterns identified at a nearby $(L=96)$ KS regime in \cite{TakeuchiEtAl11} (see Fig.~2 of that reference). Other VSA patterns also display qualitatively similar features to $\phi_{1}$ and to CLVs. While such similarities are intriguing, they should be interpreted with caution as the existence of connections between VSA and CLV techniques is an open question.

\section{\label{secConclusions}Conclusions}
           
We have presented a method for extracting spatiotemporal patterns from complex dynamical systems, which combines aspects of the theory of operator-valued kernels for machine learning with delay-coordinate maps of dynamical systems. A key element of this approach is that it operates directly on spaces of vector-valued observables appropriate for dynamical systems generating spatially extended patterns. This allows the extraction of spatiotemporal patterns through eigenfunctions of kernel integral operators with far more general structure than those captured by pairs of temporal and spatial modes in conventional eigendecomposition techniques utilizing scalar-valued kernels. In particular, our approach enables efficient and physically meaningful decomposition of signals with intermittency in both space and time, while naturally factoring out dynamical symmetries present in the data. By incorporating delay-coordinate maps, the recovered patterns lie, in the asymptotic limit of infinitely many delays, in finite-dimensional invariant subspaces of observables associated with the point spectrum of the Koopman operator of the system. This endows these patterns with high dynamical significance and the ability to decompose multiscale signals into distinct coherent modes. We demonstrated with applications to the KS model in chaotic regimes that VSA recovers intermittent patterns, such as heteroclinic orbits associated with translation families of unstable fixed points and traveling waves, with significantly higher skill than comparable eigendecomposition techniques operating on spaces of scalar-valued observables. We anticipate this framework to be applicable across a broad range of disciplines dealing with complex spatiotemporal data. 

\section*{Acknowledgments}

D.G.\ acknowledges support from NSF EAGER grant 1551489, ONR YIP grant N00014-16-1-2649, NSF grant DMS-1521775, and DARPA grant HR0011-16-C-0116. J.S.\ and A.O.\ acknowledge support from NSF EAGER grant 1551489. Z.Z.\ received support from NSF grant DMS-1521775. We thank Shuddho Das for stimulating conversations.

\appendix

\section{\label{appKoopman}Koopman operators on scalar- and vector-valued observables}

\subsection{\label{appKoopEig}Basic properties of Koopman operators and their eigenfunctions}
In this appendix, we outline some of the basic properties of the Koopman operator $ U^t $ acting on scalar-valued observables in $H_X $ and its lift $ \tilde U^t $ acting on scalar-valued observables in $ H_\Omega$ (and, by the isomorphism $ H \simeq H_\Omega $, on vector-valued observables in $H $). Additional details on these topics can be found in one of the many references in the literature on ergodic theory, e.g., \cite{BunimovichEtAl00,BudisicEtAl12,EisnerEtAl15}.  

We begin by noting that for the class of $C^1$ measure-preserving dynamical systems on manifolds studied here (see Section~\ref{secPrelim}), the group $U = \{ U^t \}_{t \in \mathbb{R}} $ of Koopman operators is a strongly continuous unitary group. This means that for every $f \in H_X $, the map $t \mapsto U^t f$ is continuous with respect to the $H_X $ norm at every $t \in \mathbb{R}$. By Stone's theorem, strong continuity of $U$ implies that there exists an unbounded, skew-adjoint operator $V : D(V) \mapsto H_X$ with dense domain $D(V) \subset H_X$, called the generator of $U$, such that $U^t = e^{tV}$. This operator completely characterizes Koopman group. Its action on an observable $f \in D(V)$ is given by
\begin{displaymath}
  V f  = \lim_{t\to0} \frac{ f \circ \Phi^t - f}{ t }, 
\end{displaymath}
where the limit is taken with respect to the $ H_X $ norm. If $ f $ is a differentiable function in $ C^1(X  )$, then $ V f = v( f ) $, where $ v $ is the vector field of the dynamics.   

A distinguished class of observables in $H_X$ are the eigenfunctions of the generator of the Koopman group. Every such eigenfunction, $z_j$, satisfies the  equation
\begin{displaymath}
  V z_j = i \alpha_j z_j,
\end{displaymath}
where $\alpha_j$ is a real frequency, intrinsic to the dynamical system. In the presence of ergodicity (assumed here), all eigenvalues of $V$ are simple, and eigenfunctions corresponding to distinct eigenvalues are orthogonal. Moreover, the  eigenfunctions can be normalized so that $\lvert z_j(x)\rvert=1$ for $\mu$-almost every $x \in X$. That is, Koopman eigenfunctions of ergodic dynamical systems can be normalized to take values on the unit circle in the complex plane, much like the functions $ e^{i\omega t} $ in Fourier analysis. 

Every eigenfunction $ z_j $ of $V$ at eigenvalue $ i \alpha_j $ is also an eigenfunction of $U^t$, corresponding to the eigenvalue $\Lambda_j^t=e^{i\alpha_jt}$. This means that along an orbit of the dynamical system, $z_j$ evolves purely by multiplication by a periodic phase factor, viz.
\begin{displaymath}
  U^tz_j(x) = z_j( \Phi^t(x)) = e^{i\alpha_jt}z_j(x),
\end{displaymath}   
where the equality holds for $ \mu $-almost every $ x \in X $. This property makes Koopman eigenfunctions highly predictable observables, which warrant identification from data. In general, the evolution of any observable $f$ lying in the closed subspace $\mathcal{D}_X = \overline{\spn\{ z_j \}}$ of $H_X $ spanned by Koopman eigenfunctions has the closed-form expansion 
\begin{equation}
  \label{eqUTFDisc}
  U^t f = \sum_j e^{i\alpha_jt} c_j z_j, \quad c_j = \langle z_j, f \rangle_{H_X}.
\end{equation}
This shows that the evolution of observables in $ \mathcal{ D}_X $ can be characterized as a countable sum of Koopman eigenfunctions with time-periodic phase factors. 

Koopman eigenvalues and eigenfunctions of ergodic systems also have an important group property; namely, if $z_j $ and $z_k$ are eigenfunctions of $V$ at eigenvalue $i \alpha_j $ and $i \alpha_k$, respectively, then the product $z_j z_k$ is also an eigenfunction, corresponding to the eigenvalue $i( \alpha_j + \alpha_k)$. Thus, the eigenvalues and eigenfunctions of the Koopman generator form groups, with addition of complex numbers and multiplication of complex-valued functions acting as the group operations, respectively. If, in addition, the eigenfunctions are continuous (which is assumed here), these groups are finitely generated. Specifically, in that case there exists a finite collection of rationally-independent eigenfrequencies $ \tilde \alpha_1, \ldots, \tilde \alpha_l $, such that every eigenfrequency has the form $ \alpha_{j} = \sum_{k=1}^l j_k \tilde \alpha_k $, where $ j = ( j_1, \ldots, j_l ) $ is a vector of integers. Moreover, the Koopman eigenfunction corresponding to eigenfrequency  $ \alpha_{ j } $ is given by $ z_{j} = \tilde z_1^{j_1} \cdots \tilde z_l^{j_l} $, where $ \tilde z_1, \ldots, \tilde z_l $ are the eigenfunctions corresponding to $ \tilde \alpha_1, \ldots, \tilde \alpha_l $, respectively. It follows from these facts in conjunction with~\eqref{eqUTFDisc} that the evolution of every observable in $\mathcal{D}_X$ can be uniquely determined given knowledge of finitely many Koopman eigenfunctions and their corresponding eigenfrequencies. 

Yet, despite these attractive properties, in typical systems, not every observable will admit a Koopman eigenfunction expansion as in~\eqref{eqUTFDisc}; that is, $\mathcal{D}_X$ will generally be a strict subspace of $H_X$. In such cases, we have the orthogonal decomposition
\begin{equation}
  \label{eqHADecomp}
  H_X = \mathcal{D}_X \oplus \mathcal{D}_X^\perp,
\end{equation}
which is invariant under the action of $ U^t $ for all $  t\in \mathbb{R}$. For observables in the orthogonal complement $ \mathcal{D}_X^\perp $ of $\mathcal{D}_X$ dynamical evolution is not determined by~\eqref{eqUTFDisc}, but rather by a spectral expansion involving a continuous spectral density (intuitively, an uncountable set of frequencies). This evolution can exhibit  the characteristic behaviors associated with chaotic dynamics, such as decay of temporal correlations. In particular, it can be shown that for any $ f\in \mathcal{ D}_X^\perp $ and $ g \in  H_X $, the quantity $ \int_0^t \lvert \langle g, U^s f \rangle_{H_X} \rvert \, ds/ t $ vanishes as $ \lvert t \rvert \to \infty $.

We now turn to the unitary group $ \tilde U = \{ \tilde U^t \}_{t\in\mathbb{R}} $ associated with the Koopman operators $ \tilde U^t $ on $ H_\Omega $. As stated in Section~\ref{secGroupSpec}, these operators are obtained by a trivial lift $ \tilde U^t = U^t  \otimes I_{H_Y} $ of the Koopman operators on $ H_X $; equivalently, we have $ \tilde U^t f = U^t \tilde \Phi^t $, where $ \tilde \Phi^t = \Phi^t \otimes I_{Y}$, and $I_Y$ is the identity map on $ Y $. The group $ \tilde U $ is generated by the densely-defined, skew-adjoint operator $ \tilde V : D( \tilde V ) \mapsto H_\Omega $,  
\begin{equation}
  \label{eqVDecomp}
    \tilde V f = \lim_{t\to0} \frac{ f \circ \tilde \Phi^t - f }{  t }, 
\end{equation}  
which is an extension of $ V \otimes I_{H_Y} $. Moreover, analogously to the decomposition in~\eqref{eqHADecomp}, there exists an orthogonal decomposition
\begin{displaymath}
  H_\Omega = \mathcal{ D}_\Omega \oplus \mathcal{ D}_\Omega^\perp, \quad \mathcal{ D}_\Omega = \mathcal{ D}_X \otimes H_Y, \quad \mathcal{ D}^\perp_\Omega = \mathcal{ D}^\perp_X \otimes H_Y,
\end{displaymath}
which is invariant under $\tilde U^t$ for all $t \in \mathbb{R}$. It is straightforward to verify that the eigenvalues of $ \tilde U^t $ are identical to those of $ U^t $, i.e., $ \Lambda^t = e^{\alpha t} $ for some eigenfrequency $ \alpha \in \mathbb{R}$, and every eigenfunction $ \tilde z $ at eigenvalue $ \Lambda^t $ has the form
\begin{equation}
  \label{eqTildeZTensor}
  \tilde z = z \otimes f,
\end{equation}
where $ z \in H_X $ is an eigenfunction of $ U^t $ at the same eigenvalue (unique up to normalization by ergodicity), and $ f $ an arbitrary spatial pattern in $ H_Y $. 

\subsection{\label{appQInf}Common eigenfunctions with kernel integral operators}

We now examine the properties of common eigenfunctions between the Koopman operators $ \tilde U^t$ on $ H_\Omega$ and the kernel integral operators $K_\infty$ from Theorem~\ref{thmKoop} with which they commute. In particular, let $ \tilde z \in H_\Omega $ be an eigenfunction of $ \tilde U^t $ at eigenvalue $\Lambda^t $. Then, 
\begin{equation}
  \label{eqUTPInf}
  \tilde U^t K_\infty \tilde z = K_\infty \tilde U^t \tilde z = \Lambda^t K_\infty \tilde z,
\end{equation}      
which implies that $ K_\infty z $ is also an eigenfunction of $ \tilde U^t $ at the same eigenvalue. As stated in Appendix~\ref{appKoopEig}, the eigenvalues of $ \tilde U^t $ are identical to those of the Koopman operator $ U^t $ on $ H_A$. However, unlike those of $ U^t $, the eigenvalues of $ \tilde U^t $ are not simple, and we cannot conclude that $ K_\infty \tilde z = \lambda \tilde z $ for some number $ \lambda $; i.e., it is not necessarily the case that $ \tilde z $ is also an eigenfunction of $ K_\infty $ (despite the fact that~\eqref{eqUTPInf} implies that every eigenspace of $ \tilde U^t$ is invariant under  $K_\infty$). In fact,  the eigenspaces of $ \tilde U^t $ are infinite-dimensional, and there is no a priori distinguished set of spatiotemporal patterns in each eigenspace. 

To identify a distinguished set of spatiotemporal patterns associated with Koopman eigenfunctions, we take advantage of the fact that $ K_\infty $ is a compact operator with finite-dimensional eigenspaces corresponding to nonzero eigenvalues. For each such eigenspace, there exists an orthonormal basis consisting of simultaneous eigenfunctions of $ K_\infty $ and $ \tilde U^t $. To verify this explicitly, let $ W_l \subset H_\Omega$ be the eigenspace of $ K_\infty $ corresponding to eigenvalue $ \lambda_l \neq 0 $, and $ f $ an arbitrary element of $ W_l $. Since 
\begin{displaymath}
  K_\infty \tilde U^t f = \tilde U^t K_\infty  f  = \lambda_l \tilde U^t f,
\end{displaymath} 
we can conclude that $ \tilde U^t f \in W_l $; i.e., that $ W_l $ is a finite-dimensional invariant subspace of $ H_\Omega $ under $ \tilde U^t $. Choosing an orthonormal basis $ \{ \phi_{1l},  \ldots, \phi_{m_ll} \} $ for this space, where $ m_l = \dim W_l $, we can expand $ f = \sum_{j=1}^{m_l} c_j \phi_{jl} $ with $ c_j = \langle \phi_{jl}, f \rangle_{H_\Omega} $, and compute 
\begin{equation}
  \label{eqUTF}
  \tilde U^t f = \sum_{i,j=1}^{m_l} \phi_{il} \tilde U_{ij} c_j, \quad \tilde U_{ij} = \langle \phi_{il}, \tilde U^t \phi_{jl} \rangle_{H_\Omega}. 
\end{equation}
By unitarity of $ \tilde U^t $, the $ m_l \times m_l $ matrix $ \mathsf U $ with elements $ \tilde U_{ij} $ is unitary, and therefore unitarily diagonalizable. Let then $ \{ v_j \}_{j=1}^{m_l} $ with $ v_j = ( v_{1j}, \ldots, v_{m_lj} )^\top $ be an orthonormal basis of $ \mathbb{ C }^{m_l} $ consisting of eigenvectors of $ \mathsf{ U } $, and $ \Lambda^t_{1l}, \ldots, \Lambda^t_{m_ll}$ be the corresponding eigenvalues. It is a direct consequence of~\eqref{eqUTF} that the set $ \{ \tilde z_{1l}, \ldots, \tilde z_{m_ll}\} $ with $ \tilde z_{jl} = \sum_{k=1}^{m_l} v_{kj} \phi_{kl} $ is an orthonormal basis of $ W_l $ consisting of Koopman eigenfunctions corresponding to the eigenvalues $ \Lambda^t_{jl} $, which much thus be given by $ \Lambda_{jl} = e^{i\alpha_{jl}t}$ for some Koopman eigenfrequency $ \alpha_{jl} \in \mathbb{R} $. Since every element of $ W_l $ is an eigenfunction of $ K_\infty $, we conclude that the $ \tilde z_{jl} $ are simultaneous eigenfunctions of $ \tilde U^t $ and $ K_\infty $. 

\section{\label{appSymmetries}Symmetry group actions}

\subsection{Basic definitions}
Let $ G $ be a topological group with a left action on the spatial domain $ Y $. By that, we mean that there exists a map $ \Gamma_Y : G \times Y \mapsto Y $ with the following properties: 
\begin{enumerate}
  \item $ \Gamma_Y $ is continuous, and $ \Gamma_Y( g, \cdot ) : Y \mapsto Y $ is a homeomorphism for all $ g \in G$.
  \item $ \Gamma_Y $ is compatible with the group structure of $G$; that is, for all $y\in Y $ and $g,g'\in G$, 
    \begin{displaymath}    
      \Gamma_Y(gg',y)= \Gamma_Y(g,\Gamma_Y(g',y)),
    \end{displaymath}
    and $ \Gamma_Y(e,y) = y $, where $e$ is the identity element of $G$.
\end{enumerate}
Given $g \in G$, we abbreviate the map $\Gamma_Y(g,\cdot) : Y \mapsto Y $ by $\Gamma_Y^g$. Note that the (continuous) inverse of this map is given by $\Gamma^{g^{-1}}_Y$.  

Assume now that $ \Gamma^g_Y$ preserves null sets with respect to the measure $ \nu $. Then, action of $G$ on $Y$ induces a continuous left action on the Hilbert space $ H_Y $ such that the action map $ \Gamma_{H_Y}^g : H_Y \mapsto H_Y $ sends $ u \in H_Y $ to $ u \circ \Gamma_Y^{g^{-1}} $. In this setting, $ G $ is considered to be a dynamical symmetry group  if the following hold for all $ g \in G $ (e.g., \cite{HolmesEtAl96}):
\begin{enumerate}
  \item The state space manifold $ X \subset H_Y $ is invariant under $ \Gamma_{H_Y}^g $. Thus, we obtain a left group action $ \Gamma_X^g $ on $ X $ by restriction of $ \Gamma_{H_Y} $.
  \item $ \Gamma^g_X $ is differentiable for all $ g \in G $, and the vector field $ v $ generating the dynamics $ \Phi^t : X \mapsto X $ is invariant under the pushforward map $ \Gamma_{X*}^g  $ associated with $ \Gamma^g_X$. That is, for every $g \in G$, $ x \in X $, and $ f \in C^1(X) $, we have 
    \begin{displaymath}
      \Gamma_{X^*}^g( v \rvert_x )( f ) = v\rvert_{\Gamma^g_X(x)}( f ),
    \end{displaymath}
    or, equivalently, 
    \begin{displaymath}
        v\rvert_x( f \circ \Gamma_X^{g} ) = v\rvert_{\Gamma^g_X(x)}( f ) . 
    \end{displaymath} 
\end{enumerate}
Note that the well-definition of $\Gamma_{X^*}^g$ as a map on vector fields relies on the fact that $\Gamma_X^g$ is a diffeomorphism (which is in turn a consequence of the fact that $\Gamma^g_X$ is a differentiable group action). 

\subsection{\label{appSymEig} Common eigenfunctions with Koopman operators}

In this section, we examine the structure of common eigenfunction between the Koopman operator $ \tilde U^t = U^t \otimes I_{H_Y} $ on $H_\Omega$ and the unitary representatives $ R^g_\Omega = R^g_X \otimes R^g_Y $ of the symmetry group $ G $ from Section~\ref{secGroupSpec}, under the assumption that $ U^t$ commutes with $ R^g_X$ (which is equivalent to $\tilde U^t$ commuting with $R^g_\Omega$). As noted in Appendix~\ref{appKoopEig}, every eigenfunction $ \tilde z $ of $\tilde U^t $ has the form $ \tilde z = z \otimes \psi $, where $ z $ is an eigenfunction of $ U^t$ corresponding to a simple eigenvalue $ \Lambda^t$, and $ \psi $ a spatial pattern in $H_Y$. As result, because $U^t$ and $R^g_X$ commute, we have
\begin{equation}
  \label{eqURA}
   U^t R^g_X  z = R^g_X U^t z = \Lambda^t R^g_X z,
\end{equation}
which shows that $ R^g_X  z $ lies in the same Koopman eigenspace as $  z $. Thus, since all eigenspaces of $ U^t $ are one-dimensional, and $R^g_X$ is unitary, there exists a complex number $ \gamma_X^g $ with $ \lvert \gamma^g_X \rvert = 1$ such that
\begin{equation}
  \label{eqRAEig}
  R^g_X z = \gamma_X^g z;
\end{equation}
in other words, $ z $ is an eigenfunction of $ R^g_X $ at eigenvalue $ \gamma_X^g $. Note that the $\gamma^g_X$ are not necessarily simple eigenvalues.  

Next, the commutativity between $ \tilde U^t$ and $ R^g_\Omega$,  in conjunction with~\eqref{eqRAEig}, leads to
\begin{displaymath}
  R^g_\Omega \tilde z =  ( R^g_X \otimes R^g_Y ) ( z \otimes \psi ) = ( \gamma_X^g z ) \otimes ( R^g_Y \psi ), 
\end{displaymath}
which implies that $ \tilde z $ is an eigenfunction of  $ R^g_\Omega $ if and only if $ \psi $ is an eigenfunction of $ R^g_Y $. The $ R^g_\Omega $ eigenvalue corresponding to $ \tilde z $ is then given by 
\begin{displaymath}
  \gamma^g_\Omega = \gamma^g_Y / \gamma^g_X, 
\end{displaymath}
where $ \gamma^g_Y $ is the $ R^g_Y $ eigenvalue corresponding to $ \psi $. Further, because $ R^g_X $ is unitary, we have 
\begin{displaymath}
  \gamma^g_\Omega =  \gamma^g_Y \gamma^{g*}_X = \gamma^g_Y \gamma^{g^{-1}}_X.
\end{displaymath}  
We have thus obtained a characterization of the common eigenspaces of $ \tilde U^t $ and $ R_\Omega^g $.

\section{\label{appVSAKernel}Construction and properties of VSA kernels}

In this appendix, we describe in detail some aspects of the VSA kernel construction, namely the choice of distance scaling function for the kernels $k_Q $ in~\eqref{eqKQ} (Appendix~\ref{appBandwidth}), and the normalization procedure to obtain the Markov kernels $p_Q$ (Appendix~\ref{appMarkov}). We also establish some results on the behavior of these kernels in the limit of infinite delays, $Q \to \infty$. Throughout, $M = A \times Y \subseteq \Omega $ will denote the compact support of the measure $\rho$.  

\subsection{\label{appBandwidth}Choice of scaling function}

The distance scaling function $ s_Q $ is based on the corresponding function introduced in \cite{Giannakis17}, which has dependencies on both the local sampling density and time tendency of the data. 

\subsubsection{\label{appDensity}Local density function}

Let $ \bar k_Q : \Omega \times \Omega \mapsto \mathbb{R} $ denote the unscaled Gaussian kernel from~\eqref{eqKGaussVSA}, and $\bar K_Q : H_\Omega \mapsto H_\Omega $ the corresponding integral operator. Following \cite{BerryHarlim16}, we define the function
\begin{equation}
  \label{eqSigmaQ}
  \sigma_Q = \bar K_Q 1_\Omega = \int_\Omega \bar k_Q( \cdot, \omega ) \, d\rho(\omega ),
\end{equation}
which is continuous, strictly positive, and bounded away from zero on compact sets. It is a standard result from the theory of kernel density estimation that if the support of the measure $ \rho $ is a smooth, $ m $-dimensional Riemannian manifold, denoted $M$, and the delay-coordinate observation map $ \tilde F_Q : \Omega \mapsto \mathbb{ R }^Q $ from~\eqref{eqFQ} is an embedding of $M $ into $ \mathbb{R}^Q$, then, as $ \epsilon \to 0 $, the quantity $ \bar \sigma_Q( \omega ) = \sigma_Q( \omega )  / ( 2 \pi \epsilon^{m/2} )  $ converges for every $ \omega \in M $ to the density $ \frac{ d\rho }{ d\text{vol} }( \omega )  $ of the measure $ \rho $  with respect to the Riemannian measure $ \text{vol} $ on $ M$. Moreover, $ \bar \sigma_Q $  has physical dimension (units) of length$^{-m} $, and as a result $ \bar \sigma_Q^{-1/m} $ assigns a characteristic length at each point in $ M $.  Here, we do not assume that $ M $ has the structure of a smooth manifold, so we will not be taking $ \epsilon \to 0 $ limits. However, due to the exponential decay of $ \bar k_Q( \omega, \omega' ) $ with respect to distance in delay-coordinate space $ \mathbb{R}^Q$, we can still interpret $ \sigma_Q $ from~\eqref{eqSigmaQ} as a local density-like function.

\subsubsection{\label{appVelocity}Phase space velocity}

Throughout this section, we will assume that the pointwise observation map $ F_y : X \mapsto \mathbb{ R } $ is of class $C^2$ for every $ y \in Y $, which is equivalent to assuming that the vector-valued observation map $\vec F$ lies in the space $C^2(X;C(Y))$. This assumption is natural for a wide class of observation maps encountered in applications. In particular, it implies that the ``energy'' of the signal, expressed in terms of the Koopman generator $ \tilde V $ from Appendix~\ref{appKoopman}  as $ \int_\Omega \lvert \tilde V( \vec  F ) \rvert^2 \, d\rho $ is finite. Under this condition, the phase space speed function $ \xi_Q : \Omega \mapsto \mathbb{ C } $, defined as 
\begin{equation}
  \label{eqXiQ}
  \xi^2_Q( x, y) = \frac{1}{Q} \sum_{q=0}^{Q-1}\zeta^2( \Phi^{-q\tau} (x), y),  \quad x \in X, \quad y \in Y,
\end{equation}
where
\begin{displaymath}
    \zeta( x, y )  = \lvert \tilde V F( x, y ) \rvert
   = \left \lvert \lim_{t\to 0} \frac{ F_y(  \Phi^t( x ) ) - F_y( x ) }{ t } \right \rvert,
\end{displaymath} 
is continuously differentiable with respect to $ x $. Note that $ \xi_Q(  x, y  ) $ may vanish, e.g., if $ y $ lies in the boundary of $ Y $ and $ F $ obeys time-independent boundary conditions. In the special case $ Q = 1 $,  $ \xi_Q $ will also vanish at local maxima/minima of the signal with respect to time. Phase space speed functions analogous to $ \xi_Q $ were previously employed in NLSA \cite{GiannakisMajda12a} and related kernel \cite{Giannakis15} and Koopman operator techniques \cite{Giannakis17}. For reasons that will be made clear below, we will adopt the approach introduced in \citep[][Section~6]{Giannakis17}, which utilizes $ \xi_Q $ in such a way so that if $ \xi_Q( \omega ) $ is zero, then $ s_Q( \omega ) $ vanishes too.

\subsubsection{\label{appScalingFn}Scaling function}

With the density and phase space speed functions from Appendices~\ref{appDensity} and~\ref{appVelocity}, respectively, we define the continuous scaling function
\begin{equation}
  \label{eqSQ}
  s_Q = \left( \sigma_Q \xi_Q \right)^\gamma, 
\end{equation}    
where $ \gamma$ is a positive parameter. This definition is motivated by \cite{Giannakis17}, where it was shown that an analogous scaling function employed in scalar-valued kernels for ergodic dynamical systems on compact Riemannian manifolds  can be interpreted, for an appropriate choice of $ \gamma $ and in a suitable limit of vanishing kernel bandwidth parameter $ \epsilon $, as a conformal change of Riemannian metric that depends on the vector field of the dynamics. More specifically, a Markov operator analogous to $ P_Q $ from Section~\ref{secMarkov} was shown to approximate the heat semigroup generated by a Laplace-Beltrami operator associated with this conformally transformed metric. In \cite{Giannakis17}, this change of geometry was associated with a rescaling of the vector field of the dynamics (i.e., a time change of the dynamical system \cite{KatokThouvenot06}) that was found to significantly improve the conditioning of kernel algorithms if the system has fixed points. In particular, for a dataset consisting of finitely many samples, the sampling of the state space manifold near a fixed point will become highly anisotropic, as most of the near neighbors of  datapoints close to the fixed point will lie along the sampled orbit of the dynamics (which is a one-dimensional set), and the directions transverse to the orbit will be comparatively undersampled. The latter is because the phase speed of the system becomes arbitrarily small near a fixed point, meaning that most geometrical  nearest neighbors of a data point in its vicinity will lie on a single orbit. Choosing the scaling function (analogous to $s_Q$) such that it vanishes at the fixed point is tantamount to increasing the bandwidth $ 1/ s_Q $ of the kernel  by arbitrarily large amounts, thus improving sampling in directions transverse to the orbit. 

While the arguments above are strictly valid only in the smooth manifold case (as they rely on $ \epsilon \to 0 $ limits),  $ s_Q $ in~\eqref{eqSQ} should behave similarly in regions of the product space $\Omega $ where the rate-of-change of the observed data (measured by $ \xi_Q $ in~\eqref{eqXiQ}) is small.  As stated in Section~\ref{secVSAKernel}, $ \xi_Q $ can vanish or be small not only at fixed points of the dynamics on $ X $, but also at points  $y \in Y $ where the observable $F_y $  is constant or nearly constant (e.g., near domain boundaries). 

What remains for a complete specification of $ s_Q $ is to set the exponent parameter $ \gamma $. According to \cite{Giannakis17}, if $ M $ is an $m $-dimensional manifold embedded in $ \mathbb{R}^Q$, the Riemannian metric associated with $ P_Q $ for the choice $ \gamma = 1 / m $ has compatible volume form with the invariant measure of the dynamics, in the sense that the corresponding density $ \frac{d\rho}{d \text{vol}} $ is a constant. Moreover, the induced metric is also invariant under a class of conformal changes of observation map $F_Q$.  In practice, $ M $ will not be a smooth manifold, but we can still assign to it an effective dimension by examining the dependence of the kernel integrals $ \kappa = \int_{\Omega\times \Omega } \bar k_Q \, d \rho \times d\rho $ (or the corresponding data-driven quantity $ \kappa_{NS} = \int_{\Omega \times \Omega} \bar k_Q \, d\rho_{NS} \times d\rho_{NS} $, where the measures $ \rho_{NS} $ are defined in Section~\ref{secDataDriven}) as a function of the bandwidth parameter $ \epsilon $. As shown in \cite{CoifmanEtAl08,BerryHarlim16}, $  d \log \kappa / d \log \epsilon $ can be interpreted as an effective dimension for  at the scale associated by the bandwidth parameter $ \epsilon $. This motivates an automatic bandwidth tuning procedure where $ \epsilon $ is chosen as the maximizer of that function, and the corresponding maximum value $ \hat m $ provides an estimate of $ M $'s dimension. 

Here, we nominally set $ \gamma = 1 / \hat m $ with $ \hat m $ determined via the method just described. The results presented in Section~\ref{secApplications}  are not too sensitive with respect to changes of $ \gamma $ around that value. In fact, for the systems studied here, the results remain qualitatively robust even if the velocity-dependent terms are not included in $ s_Q $ and $ s_{Q,N} $. That is, qualitatively similar results can also be obtained using the scaling function $ s_Q = \sigma_Q$, which is continuous even if $ F $ is not continuously differentiable.

\subsection{\label{appMarkov}Markov normalization}

Following the approach taken in NLSA algorithms and in \cite{BerryEtAl13,GiannakisEtAl15,Giannakis17,DasGiannakis17}, we will construct a Markov kernel $ p_Q : \Omega \times \Omega \mapsto \mathbb{R} $ from a strictly positive, symmetric kernel $ k_Q : \Omega \times \Omega \mapsto \mathbb{R} $, meeting the conditions in Section~\ref{secKernelDelays}, by applying the normalization procedure introduced in the diffusion maps algorithm \cite{CoifmanLafon06} and further developed in the context of general exponentially decaying kernels in \cite{BerrySauer16b}. For that, we first compute the functions
\begin{displaymath}
  r_Q = K_Q 1_\Omega, \quad l = K_Q( 1_\Omega / r_Q )
\end{displaymath}
where $ 1_\Omega $ is the function on $ \Omega$ equal to 1 at every point. By the properties of $ k_Q$ and compactness of the support of $ \rho$, both $ r_Q $ and $l_Q $ are continuous, positive functions on $ \Omega $, bounded away from zero on compact sets. We then define the kernel $ p_Q $ by
\begin{equation}
  \label{eqPKernel}
  p_Q( \omega, \omega' ) = \frac{ k_Q( \omega, \omega' ) }{ l_Q( \omega ) r_Q( \omega' ) },
\end{equation} 
and the Markov property follows by construction. In \cite{BerrySauer16b}, the division of $ k_Q( \omega, \omega' ) $ by $ l_Q( \omega ) $ and $ r_Q( \omega' ) $ to form $ p_Q( \omega, \omega' ) $ is referred to as left and right normalization, respectively. Because $ r_Q $ and $ l_Q $ are positive and bounded away from zero on compact sets, $ p_Q $ is continuous.

In general, a kernel of the class in~\eqref{eqPKernel}, is not symmetric, and as a result the corresponding integral operator $ P_Q : H_\Omega \mapsto H_\Omega $  is not self-adjoint. Nevertheless, by symmetry of $ k_Q $, $ P_Q $ is related to a self-adjoint compact operator $ \hat P_Q : H_\Omega \mapsto H_\Omega $ by a similarity transformation. In particular, let $ f $ be a bounded function in $ L^\infty( \Omega, \rho ) $, and $ T_f : H_\Omega \mapsto H_\Omega $ the corresponding multiplication operator by $ f $. That is, for $ g \in H_\Omega $, $T_f g $ is the function equal to $ f( \omega ) g( \omega ) $ for $ \rho $-a.e.\ $ \omega \in\Omega $. Defining $ \hat P_Q : H_\Omega \mapsto H_\Omega $ as the self-adjoint kernel integral operator associated with the symmetric kernel
\begin{displaymath}
  \hat p_Q( \omega, \omega' ) = \frac{ k_Q( \omega, \omega' ) }{ \hat l_Q( \omega ) \hat l_Q( \omega' ) }, \quad \hat l_Q = \sqrt{ l_Q r_Q }, 
\end{displaymath} 
one can verify that $ \hat P_Q $ can be obtained from $P_Q $ through the similarity transformation
\begin{equation}
  \label{eqPSimilarity}
  \hat P_Q = T_{\tilde l_Q} \circ P_Q \circ T_{\tilde l_Q}^{-1}, \quad \tilde l_Q = \sqrt{ l_Q / r_Q }. 
\end{equation}  
Due to~\eqref{eqPSimilarity},  $ P_Q $ and $ \hat P_Q $ have the same eigenvalues $ \lambda_j $, which are real by self-adjointness of $ \hat P_Q $, and thus admit the ordering  
$ 1 = \lambda_0 > \lambda_1 \geq \lambda_2 \geq \cdots $ since $ P_Q $ is ergodic and Markov. Moreover, by compactness of $ P_Q $ and $ \hat P_Q $, the eigenvalues have a single accumulation point at zero, and the eigenspaces corresponding to the nonzero eigenvalues are finite-dimensional.

Since $ \hat P_Q $ is self-adjoint, there exists an orthonormal basis $ \{ \hat \phi_j \}_{j=0}^\infty $ of $ H_\Omega $ consisting of real eigenfunctions $ \hat \phi_j $ of $ \hat P_Q $ corresponding to  $ \lambda_j $, which are continuous by the assumed continuity of kernels. Moreover, due to~\eqref{eqPSimilarity}, for every element $ \hat \phi_j $ of this basis, the continuous functions $ \phi_j = T_{\tilde l_Q} \hat \phi_j = \tilde l_Q \hat \phi_j $ and $ \phi'_j = T_{\tilde l_Q}^{-1} \hat \phi_j = \hat \phi_j / \tilde l_Q $ are eigenfunctions of $  P_Q $ and $  P_Q^* $, respectively, corresponding to the same eigenvalue $ \lambda_j $. The sets $ \{ \phi_j \}_{j=0}^\infty $ and $ \{ \phi'_j \}_{j=0}^\infty $ are then (non-orthogonal) bases of $ H_\Omega $ satisfying the bi-orthogonality relation $ \langle \phi'_i, \phi_j \rangle_{H_\Omega} = \delta_{ij} $. In particular, every $ f \in H_\Omega $ can be uniquely expanded as $ f = \sum_{j=0}^\infty c_j \phi_j $ with $ c_j = \langle \phi'_j, f \rangle_{H_\Omega} $, and we have $ P_Q f = \sum_{j=0}^\infty \lambda_j c_j \phi_j $. 

\subsection{\label{appKDelay}Behavior in the infinite-delay limit}

In this section we establish that the covariance and Gaussian kernels $ k_Q $ in~\eqref{eqKCovVSA}--\eqref{eqKQ}, as well as the Markov kernels in Section~\ref{secMarkov}, converge to well-defined, shift-invariant limits in the infinite-delay ($Q\to\infty$) limit, in accordance with the conditions for VSA kernels listed in Section~\ref{secKernelDelays}. Since all of these kernels are based on distances between datapoints in delay-coordinate space under the maps $\tilde F_Q$ from~\eqref{eqFQ}, we begin by considering the family of distance-like functions $ d_Q : \Omega \times \Omega \mapsto \mathbb{R}$, with $Q \in \mathbb{N}$ and
\begin{displaymath}
    d_Q( \omega, \omega' ) = \frac{1}{Q^{1/2}}\lVert \tilde F_Q(\omega) - \tilde F_Q(\omega') \rVert_{\mathbb{R}^Q}.
\end{displaymath}
This family of functions has the following important property: 
\begin{prop}
    \label{propDInf}
    Suppose that the sampling interval $ \tau $ is such that there exists no eigenfrequency $ \alpha_j $ of the generator $ V $ such that $ e^{i\alpha_j \tau} = 1 $. Then, the sequence $ d_1, d_2, \ldots $ converges in $ H_\Omega \times H_\Omega $ norm to a $ \tau$-independent limit  $ d_\infty \in H_\Omega \otimes H_\Omega $, satisfying $ \tilde U^t \otimes \tilde U^t d_{\infty} = d_\infty $ for all $ t \in \mathbb{R}$.   
\end{prop}
\begin{proof}
    Let $ \omega = ( x, y)$ and $ \omega' = ( x', y' ) $ with $ x, x' \in X$ and $ y, y' \in Y$ be arbitrary points in $ \Omega$. The $ H_\Omega \otimes H_\Omega $ convergence of $ d_Q $ to $ d_\infty $ follows from the Von Neumann mean ergodic theorem and the fact that  
\begin{displaymath}
    d^2_Q( \omega, \omega' ) = \frac{1}{Q} \sum_{q=0}^{Q-1} \lvert F_y(\Phi^{-q\tau}(x)) - F_{y'}(\Phi^{-q\tau}(x')) \rvert^2 = \frac{1}{Q} \sum_{q=0}^{Q-1} d_1^2( \Phi^{-q\tau}(\omega), \Phi^{-q\tau}(\omega')), 
\end{displaymath}
is a Birkhoff average under the product dynamical system $ \tilde \Phi^{q \tau} \otimes \tilde \Phi^{q \tau} $ on $ \Omega \times \Omega$ of the function $ d_1^2 : \Omega \times \Omega \mapsto \mathbb{R} $, which is bounded on the compact support of the corresponding invariant measure $ \rho \times \rho $. 

Next, let $  \{ \tilde U^t \otimes U^t \}_{t\in\mathbb{R}} $ be the strongly continuous unitary group induced by $ \tilde \Phi^t \times \tilde \Phi^t $. Denote the generator of this group by $ \hat V $. To establish $ \tau$-independence and invariance of $ d_\infty $ under $ \tilde U^t \otimes \tilde U^t$, it suffices to show that $ d_\infty$ lies in the nullspace of $ \hat V $.  Indeed, by invariance of infinite Birkhoff averages, we have $ \tilde U^\tau \otimes \tilde U^\tau d_\infty = d_\infty $, i.e., $ d_\infty $ is an eigenfunction of $ \tilde U^\tau \otimes \tilde U^\tau $ at eigenvalue 1, and by the condition on $ \tau $ stated in the proposition and the fact that $ U^t $, $ \tilde U^t $, and $ \tilde U^t \otimes \tilde U^t $ have the same eigenvalues (see Appendix~\ref{appKoopman}), this implies that $ d_\infty $ is an eigenfunction of $ \hat V $ at eigenvalue zero.
\end{proof}

Note that the condition on $ \tau $ in Proposition~\ref{propDInf} holds for Lebesgue almost every $ \tau \in \mathbb{R}$, as the set of Koopman eigenfrequencies $ \alpha_j $ is countable. 

An immediate consequence of Proposition~\ref{propDInf} is that given any continuous shape function $ h : \mathbb{R} \mapsto \mathbb{R} $, the kernel $ k_Q : \Omega \times \Omega \mapsto \mathbb{R} $ with $ k_Q( \omega, \omega' ) = h( d_Q( \omega, \omega' ) )$ satisfies the conditions listed in Section~\ref{secKernelDelays}. In particular, setting $h$ to a Gaussian shape function, $ h(s) = e^{-s^2/\epsilon}$, with $ \epsilon > 0$, shows that the Gaussian kernel in~\eqref{eqKGaussVSA} has the desired properties. That the covariance kernel in~\eqref{eqKCovVSA} also has these properties follows from an analogous result to Proposition~\ref{propDInf} applied directly to the kernel $ k_Q $, which in this case is equal to a Birkhoff average of $ k_1 $.    

Next, we turn to the family of kernels in~\eqref{eqKQ} utilizing scaled distances. These kernels have the general form 
\begin{displaymath}
    k_Q(\omega,\omega') = h( s_Q( \omega ) s_Q( \omega' ) d_Q( \omega, \omega' ) ),
\end{displaymath}
where $ h : \mathbb{R} \mapsto \mathbb{R } $ is continuous, so by Proposition~\ref{propDInf}, the required properties will follow if it can be shown that the sequence of scaling functions $ s_1, s_2, \ldots $ converges in $ H_\Omega $ norm to a bounded function $ s_\infty \in L^\infty(\Omega,\rho) $. That this is indeed the case for the choice of scaling functions described in Appendix~\ref{appBandwidth} follows from the facts that (i) the local density function $ \sigma_Q $ in~\eqref{eqSigmaQ} is itself derived from an unscaled Gaussian kernel $ \bar k_Q$, which was previously shown to meet the required conditions, and (ii) the phase space velocity function $ \xi_Q $ from~\eqref{eqXiQ} is equal to a Birkhoff average of a continuous function. 

Finally, the class of Markov kernels from Section~\ref{secMarkov} meets the necessary conditions because the diffusion maps normalization function $ r_Q $ is a continuous function  determined by action of $ K_Q $ on a constant function, thus converging as $Q \to \infty $ to a $ \tilde U^t$-invariant function by the previous results on $k_Q$, and similarly $ l_Q$ is determined by action of $ K_Q $ on $ 1/ r_Q $ (see Appendix~\ref{appMarkov}). 

\section{\label{appDataDriven}Data-driven approximation}

In this appendix, we present a proof of Theorem~\ref{thmDataDriven} for the most general class of integral operators on $H_\Omega$ studied in this paper, namely the Markov operators $P_Q$ associated with the kernels $ k_Q: \Omega \times \Omega \mapsto \mathbb{R} $ from~\eqref{eqKQ} utilizing the distance scaling functions $ s_Q $  in~Appendix~\ref{appBandwidth}, followed by Markov normalization as described in Section~\ref{secMarkov} and Appendix~\ref{appMarkov}. The convergence results for the operators not employing distance scaling and/or Markov normalization follow by straightforward modification of the arguments below. 

\subsection{Data-driven Markov kernels}

Because the distance scaling functions $ s_Q $ involve integrals with respect to $ \rho $ and time derivatives, in a data-driven setting we must work with kernels  $ k_{Q,NS} : \Omega \times \Omega \mapsto \mathbb{ R }_+ $ approximating $ k_Q $,  where
\begin{displaymath}
    k_{Q,NS}(\omega,\omega') = \exp \left( - \frac{s_{Q,NS}(\omega) s_{Q,NS}(\omega')}{\epsilon Q} \sum_{q=0}^{Q-1} \left\lvert F_y( \Phi^{q \tau}(x)) - F_{y'}(\Phi^{q \tau}(x')) \right\rvert^2 \right), 
\end{displaymath}
and  $ s_{Q,NS} \in C( \Omega ) $ are scaling functions approximating $ s_Q$.

Our construction of $ s_{Q,NS}$ follows closely that of $ s_Q$ in~\eqref{eqSQ}; that is, we set 
\begin{displaymath}
    s_{Q,NS} = ( \sigma_{Q,NS} \xi_{Q,N})^\gamma,
\end{displaymath}
where $\gamma$ is the same exponent parameter as in~\eqref{eqSQ}, and $ \sigma_{Q,NS}$ and $ \xi_{Q,N}$ are continuous functions approximating $ \sigma_Q$ and $ \xi_Q$, respectively. To construct $ \sigma_{Q,NS} $, we introduce the integral operator $ \bar K_{Q,NS} : H_{\Omega,NS}\mapsto H_{\Omega,NS} $ with 
\begin{displaymath}
    \bar K_{Q,NS} f(\omega) = \int_\Omega \bar k_Q( \omega, \omega' ) f( \omega' ) \, d\rho_{NS}( \omega' ), 
\end{displaymath}
and define
\begin{equation}
    \label{eqSigmaQN} \sigma_{Q,NS} = \bar K_{Q,NS} 1_\Omega = \frac{ 1 }{ N } \sum_{n=0}^{N-1} \sum_{s=0}^{S-1} \bar k_Q( \cdot, \omega_{ns} ) w_{s,S}. 
\end{equation}  
Moreover, following \cite{GiannakisMajda12a,Giannakis15,Giannakis17}, in the data-driven setting we approximate the function $ \zeta $ used in the definition of $ \xi_Q$ in~\eqref{eqXiQ} by a continuous function $ \zeta_\tau : \Omega \mapsto \mathbb{ R } $ that provides a finite-difference approximation of $ \zeta $ with respect to the sampling interval $ \tau $. As a concrete example, we consider a backward difference scheme,
\begin{displaymath}
    \zeta_{\tau}(  x, y  ) = \frac{ \left \lvert F_y( x ) - F_y( \Phi^{-\tau}( x )) \right \rvert  }{ \tau }, \quad x \in X, \quad y \in Y,
\end{displaymath}
which, under the assumed differentiability properties of $ F_y $ (see Appendix~\ref{appVelocity}), converges as $ \tau \to 0 $ to $ \zeta $, uniformly on compact sets (in particular, $ \mathcal{V} $). We will also consider that the sampling interval is specified as a function $ \tau( N ) $ such that $ \tau( N ) \to 0 $ and $ N \tau( N ) \to \infty $ as $ N \to \infty $. With this assumption, the limit $ N \to \infty $ corresponds to infinitely short sampling interval (required for convergence of finite different schemes) and infinitely long total sampling time (required for convergence of ergodic averages). We define $ \xi_{Q,N} $ as the function resulting by substituting $\zeta$ by $\zeta_{\tau(N)}$ in~\eqref{eqXiQ}. As we will establish in Appendix~\ref{appClaimB}, with these definitions, and under weak convergence of the measures $\rho_{NS}$ in~\eqref{eqPhysicalM}, $s_{Q,NS} $ converges to $ s_Q $ uniformly on the compact set $\mathcal{V}$.

Next, we construct the Markov kernels $ p_{Q,NS} : \Omega \times \Omega \mapsto \mathbb{R} $ of the operators $ P_{Q,NS} : H_{\Omega,NS} \mapsto H_{\Omega,NS} $ in the statement of the theorem by applying diffusion maps normalization to $ k_{Q,NS} $ as in Appendix~\ref{appMarkov}, viz.
\begin{equation}
  \label{eqPNKernel}
  p_{Q,NS}( \omega, \omega' ) = \frac{ k_{Q,NS}( \omega, \omega' ) }{ l_{Q,NS}( \omega ) r_{Q,NS}( \omega' ) }, \quad r_{Q,NS} = K_{Q,NS} 1_\Omega, \quad l_{Q,NS} = K_{Q,NS}( 1_\Omega / r_{NS} ),
\end{equation} 
where $K_{Q,NS} $ is the kernel integral operator on $H_{\Omega,NS}$ associated with $ k_{Q,NS}$. Note that  $ r_{Q,NS} $ and $ l_{Q,NS} $ are positive, continuous functions on $ \Omega $, bounded away from zero on compact sets. 

\subsection{\label{appConv}Proof of Theorem~\ref{thmDataDriven}}

We will need the important notion of \emph{compact convergence} of operators on Banach spaces \cite{VonLuxburgEtAl08,Chatelin11}.
\begin{defn}
  \label{defCC}A sequence of bounded operators $ T_n : E \mapsto E $ on a Banach space $ E $ is said to converge compactly to a bounded operator $ T : E \mapsto E $ if $ T_n $ converges to $ T $ pointwise (i.e., $ T_n f \to T f $ for all $ f \in E $), and for every bounded sequence of vectors $ f_n \in E $, the sequence $ g_n = ( T_n - T ) f_n $ has compact closure (equivalently, $ g_n $ has a convergent subsequence).
\end{defn}
Compact convergence is stronger than pointwise convergence, but weaker than convergence in operator norm. For our purposes, it is useful as it is sufficient to imply convergence of isolated eigenvalues of bounded operators \cite{Chatelin11}, and hence convergence of nonzero eigenvalues of compact operators and their corresponding eigenspaces. In particular, Theorem~\ref{thmDataDriven} is a corollary of the following theorem, proved in Appendix~\ref{appConv2}:

\begin{thm}
    \label{thmConv}Under the assumptions of Theorem~\ref{thmDataDriven}, the following hold:
  \begin{enumerate}[(a)]
      \item  $ \tilde P_{Q,NS} $ and $ \tilde P_Q $ are both compact operators on $ C( \mathcal{V} ) $. As a result, their nonzero eigenvalues have finite multiplicities, and accumulate only at zero.
      \item As $ N,S \to \infty $, $  \tilde P_{Q,NS} $ converges compactly to $ \tilde P_Q $. 
      \item $ \lambda_j $ is a nonzero eigenvalue of $\tilde P_Q$  if and only if it is a nonzero eigenvalue of $  P_Q $. Moreover, if $ \phi_{j} $ is an eigenfunction of $ P_Q $ corresponding to that eigenvalue, then $ \tilde \phi_j \in C(\mathcal{V}) $ with
          \begin{equation}
              \label{eqTildePhi}
              \tilde \phi_j( \omega ) = \int_\Omega p_Q(\omega,\omega') \, d\rho(\omega') 
          \end{equation}
          is an eigenfunction of $ \tilde P_Q $ corresponding to the same eigenvalue. Analogous results hold for every nonzero eigenvalue of $ \lambda_{j,NS} $, of $ \tilde P_{Q,NS} $, and corresponding eigenfunctions $\phi_{j,NS} $ and $ \tilde \phi_{j,NS} $ of $ P_{Q,NS} $ and $\tilde P_{Q,NS}$, respectively, where
          \begin{equation}
              \label{eqTildePhiN}
              \tilde \phi_{j,NS}(\omega) = \int_\Omega p_{Q,NS}(\omega,\omega') \, d\rho_{NS}(\omega').
          \end{equation}
  \end{enumerate}
\end{thm}
To verify that Theorem~\ref{thmConv} indeed implies Theorem~\ref{thmDataDriven} (with $K_Q$ replaced by $P_Q$ and $ K_{Q,NS}$ by $P_{Q,NS}$), note that since the eigenvalue $ \lambda_j $ in the statement of Theorem~\ref{thmDataDriven} is nonzero, it follows from Theorem~\ref{thmConv}(c) that it is an eigenvalue of $ \tilde P_Q $, and that $ \tilde \phi_j $ from~\eqref{eqTildePhi} is a corresponding eigenfunction. Moreover, since, by Theorem~\ref{thmConv}(b), $ \tilde P_{Q,NS} $ converges to $ \tilde P_Q $ compactly (and thus in spectrum for nonzero eigenvalues \cite{Chatelin11}), there exist $ N_0, S_0 \in \mathbb{ N } $ such that the $ j $-th eigenvalues $ \lambda_{j,NS} $ of $ \tilde P_{Q,NS} $ are all nonzero for $ N \geq N_0 $ and $ S \geq S_0$, and thus, by Theorem~\ref{thmConv}(c), they are eigenvalues of $ P_{Q,NS} $ converging to $\lambda_j$, as claimed in Theorem~\ref{thmDataDriven}. The existence of eigenfunctions $ \phi_{j,NS} $ of $ P_{Q,NS}$  corresponding to $ \lambda_{j,NS} $, such that $ \tilde \phi_{j,NS} $ from~\eqref{eqTildePhiN} converges uniformly to $ \phi_j $ is shown in \cite{DasGiannakis17}. This completes our proof of Theorem~\ref{thmDataDriven}.   

\subsection{\label{appConv2}Proof of Theorem~\ref{thmConv}}

Our proof of Theorem~\ref{thmConv} draws heavily on the spectral convergence results  on data-driven kernel integral operators established in \cite{VonLuxburgEtAl08,DasGiannakis17}, though it requires certain modifications appropriate for the class of kernels in~\eqref{eqKQ} utilizing scaled distances, which, to our knowledge, have not been previously discussed. In what follows, we provide explicit proofs of Claims (a) and (b) of the theorem; Claim~(c) is a direct consequence of the definition of $ \tilde \phi_{j,NS} $ in~\eqref{eqTildePhiN}. Throughout this section, all operators will act on the Banach space of continuous functions on $ \mathcal{V} $ equipped with the uniform norm, $ \lVert \cdot \rVert_{C(\mathcal{V})} $. Therefore, for notational simplicity, we will drop the tildes from our notation for $ \tilde P_Q $ and $ \tilde P_{Q,NS} $. We will also drop $S$ subscripts representing the number of sampled points in the spatial domain $Y$,with the understanding that $ N \to \infty $ limits correspond to $ S \to \infty$ followed by $ N \to \infty$ limits. 
       
\subsection{Proof of Claim (a)}

Let $ \bar d : \Omega \times \Omega \mapsto \mathbb{R} $ be any metric on $\Omega$. We begin by establishing the following result on the kernel $p_Q$:
\begin{lem}
    The map $ \omega \mapsto p_Q(\omega, \cdot) $ is a continuous map from $\mathcal{V}$ to $C(\mathcal{V})$; that is, for any $ \epsilon > 0 $, there exists $\delta > 0 $ such that for all $\omega,\omega' \in \mathcal{V}$ satisfying $\bar d(\omega, \omega') < \delta$, 
    \begin{displaymath}
    \lVert p_{Q}(\omega,\cdot) - p_Q(\omega',\cdot) \rVert_{C(\mathcal{V})} < \epsilon.
\end{displaymath}
    \label{lemmaPQUniform}
\end{lem}
\begin{proof}
    Suppose that the claim is not true. Then, there exists $ \epsilon > 0 $ and  sequences $\omega_n,\omega'_n \in \mathcal{V}$,  such that, as $n\to\infty$, $\bar d(\omega_n,\omega'_n) \to 0$ and $\lVert p_{Q}(\omega_n,\cdot) - p_Q(\omega'_n,\cdot) \rVert_{C(\mathcal{V})} > \epsilon$. As a result, there exists $\omega''_n \in \mathcal{V}$ such that $\lvert p_Q(\omega_n,\omega''_n)-p_Q(\omega'_n,\omega''_n) \rvert > \epsilon$. However, this contradicts the fact that $p_Q$ is continuous since $(\omega_n,\omega''_n) \in \mathcal{V} \times \mathcal{V}$ converges to $(\omega'_n,\omega''_n)$.
\end{proof}

We now return to the proof of Claim~(a). First, that $ P_{Q,N} $ is compact follows immediately from the fact that it has finite rank. Showing that $ P_Q $ is compact is equivalent to showing that for any bounded sequence $ f_n \in C( \mathcal{V} ) $, the sequence  $ g_n = P_Q f_n$ has a limit point in the uniform norm topology. Since $ \mathcal{V} $ is compact, it suffices to show that $g_n $ is equicontinuous and bounded; in that case, the existence of a limit point of $ g_n $ is a consequence of the Arzel\`a-Ascoli theorem. Indeed, for any $ \omega \in \mathcal{V} $, we have
\begin{align*}
  \lvert g_n( \omega ) \rvert &= \left \lvert \int_\Omega p_Q( \omega, \omega' ) f_n( \omega' ) \, d\rho( \omega' ) \right \rvert \\
  & \leq \int_\Omega \lvert p_Q( \omega, \omega' ) f_n( \omega' ) \rvert \, d\rho( \omega' ) \\
  & \leq \lVert p_Q \rVert_{C(\mathcal{V}\times\mathcal{V})} \lVert f_n \rVert_{C(\mathcal{V})} \\
  & \leq \lVert p_Q \rVert_{C(\mathcal{V}\times\mathcal{V})} B,
\end{align*}
where $ B = \sup_n \lVert f_n \rVert_{C(\mathcal{V})} $. This shows that $ g_n $ is uniformly bounded. Similarly, we have
\begin{displaymath}
    \lvert g_n( \omega ) - g_n( \omega' ) \rvert \leq \lVert p_Q( \omega, \cdot ) - p_Q( \omega', \cdot ) \rVert_{C(\mathcal{V})} \lVert f_n \rVert_{C(\mathcal{V})}, 
\end{displaymath}
and the equicontinuity of $\{ g_n \} $ follows from Lemma~\ref{lemmaPQUniform}. It therefore follows from the Arzel\`a-Ascoli theorem that $g_n$ has a limit point, and thus  that $ P_Q $ is compact, as claimed. 

\subsection{\label{appClaimB}Proof of Claim (b)}

According to Definition~\ref{defCC}, we must first show that for every $ f \in C( \mathcal{V} ) $, $ P_{Q,N} f $ converges to $ P_Q f $ in the uniform norm; that is, we must show that $ \lim_{N\to\infty} \eta_N = 0 $, where 
\begin{displaymath}
  \eta_N = \lVert P_{Q,N} f  - P_Q f \rVert_{C(\mathcal{V})}. 
\end{displaymath}
Defining the Markov kernel $ \hat p_{Q,N} : \Omega\times \Omega \mapsto \mathbb{ R }_+ $,
\begin{equation}
  \label{eqPNHatKernel}
  \hat p_{Q,N}(\omega,\omega') = \frac{ k_Q(\omega,\omega') }{ l_{Q,N}(\omega) r_{Q,N}(\omega') },
\end{equation} 
and the operators $ \tilde P_{Q,N} : C( \mathcal{V} ) \mapsto C( \mathcal{V} ) $ with
\begin{displaymath}
  \tilde P_{Q,N} f = \int_\Omega  p_Q( \cdot, \omega ) f( \omega) \, d\rho_N( \omega ), \quad
  \hat P_{Q,N} f = \int_\Omega \hat p_{Q,N}( \cdot, \omega ) f( \omega ) \, d\rho_N( \omega ),
\end{displaymath}
we have
\begin{equation}
  \label{eqEtaN} 
   \eta_N   \leq  \lVert P_{Q} f - \tilde P_{Q,N} f \rVert_{C(\mathcal{V})} \\
   + \lVert \tilde P_{Q,N} f - \hat P_{Q,N} f \rVert_{C(\mathcal{V})} \\
   + \lVert \hat P_{Q,N} f - P_{Q,N} f \rVert_{C(\mathcal{V})}.
\end{equation} 
That is, we can bound $ \eta_N $ by the sum of contributions due to (i) errors in approximating integrals with respect to the invariant measure $ \rho $  by the sampling measure $ \rho_N $ (the first term in the right-hand side); (ii) errors in approximating the left and right normalization functions $ l_Q $ and $ r_Q $ by their data-driven counterparts, $ l_{Q,N} $ and $ r_{Q,N} $, respectively (the second term in the right-hand side); and (iii) errors in approximating the kernel $ k_{Q} $ by the data-driven kernel $ k_{Q,N} $ (the third term in the right-hand side). 

We first consider the first term,
\begin{displaymath}
  \lVert P_{Q} f - \tilde P_{Q,N} f \rVert_{C(\mathcal{V})} = \max_{\omega \in \mathcal{V}} \lvert \tilde P_{Q,N} f( \omega ) - P_Q f( \omega ) \rvert.
\end{displaymath}
By the weak convergence of the measures $ \rho_N$ to $\rho$ (see~\eqref{eqPhysicalM}) in conjunction with the continuity of $ p_Q $, it follows that $ \tilde P_{Q,N} f( \omega ) $ converges to $ P_N f( \omega )$, pointwise with respect to $ \omega \in \mathcal{V} $; however, it is not necessarily the case that the convergence is uniform. For the latter, we need the stronger notion of a \emph{Glivenko-Cantelli class}.
\begin{defn}
  Let $ \mathbb{E} : C( \mathcal{V} ) \mapsto \mathbb{ C } $ and $ \mathbb{E}_N : C( \mathcal{V} ) \mapsto \mathbb{ C } $, be the expectation operators with respect to the measures $ \rho_N $ and $ \rho $, respectively, i.e., 
\begin{displaymath}
  \mathbb{E} f = \int_\Omega f \, d\rho, \quad \mathbb{E}_N f = \int_\Omega f\, d\rho_N, \quad f \in C( \mathcal{V} ).
\end{displaymath}
Then, a set of functions $ \mathcal{ F } \in C( \mathcal{V} ) $ is said to be a Glivenko-Cantelli class if 
\begin{displaymath}
  \lim_{N\to\infty} \sup_{f\in\mathcal{F}} \lvert \mathbb{E} f  - \mathbb{E}_N f \rvert = 0. 
\end{displaymath}
\end{defn}
Note, in particular, that if the set 
\begin{displaymath}
  \mathcal{ F }_1 = \{ p_Q( \omega, \cdot ) f( \cdot ) \mid \omega \in \mathcal{V} \} 
\end{displaymath}
can be shown to be a Glivenko-Cantelli class, then it will follow that $ \lVert \tilde P_{Q,N} f - P_Q f \rVert_{C(\mathcal{V})} $ vanishes as $ N \to \infty $. That this is indeed the case follows from Proposition~11 in \cite{VonLuxburgEtAl08}.

Next, we turn to the second and third terms in~\eqref{eqEtaN}. To bound these terms, we first establish convergence of the data-driven distance scaling functions $ s_{Q,N} $ to $ s_Q $.  
\begin{lem}
    \label{lemScalingConv}Restricted to $ \mathcal{V} $, the scaling functions $ s_{Q,N} $ from Appendix~\ref{appBandwidth} converge uniformly as $ N \to \infty $ to $ s_Q $.
\end{lem} 
\begin{proof}
  It follows from the definition of $ s_Q $ and $ s_{Q,N} $ in~\eqref{eqSQ} that for all $ \omega \in \mathcal{V} $, 
\begin{align*}
  \lvert s_{Q,N}( \omega ) - s_Q( \omega ) \rvert &= \lvert \sigma_{Q,N}( \omega ) \xi_{Q,N}( \omega ) - \sigma_Q( \omega ) \xi_Q( \omega ) \rvert^\gamma \\
  &\leq ( \lvert \sigma_{Q,N}(\omega) - \sigma_Q(\omega) \rvert \lvert \xi_{Q,N}( \omega ) \rvert + \lvert \sigma_Q( \omega ) \rvert \lvert \xi_{Q,N}(\omega) - \xi_Q( \omega ) \rvert )^\gamma.
\end{align*}
Thus, since $ \xi_{Q,N} $ converges uniformly to $ \xi_Q $ by continuous differentiability of the observation map $ F $ on the compact set $ \mathcal{V} $ (see Appendix~\ref{appVelocity}), $ s_{Q,N} $ will converge uniformly to $ s_Q $ if $ \sigma_{Q,N} $ converges uniformly to $ \sigma_Q $. Indeed, because 
\begin{displaymath}
 \lvert \sigma_{Q,N}( \omega ) - \sigma_Q( \omega ) \rvert = \lvert \mathbb{ E }\bar k_Q( \omega, \cdot ) - \mathbb{ E }_N \bar k_Q( \omega, \cdot ) \rvert, 
\end{displaymath}
this will be the case if the set
\begin{displaymath}
  \mathcal{ F }_2 = \{ \bar k_Q( \omega, \cdot ) \mid \omega \in \mathcal{V} \}
\end{displaymath}
is a Glivenko-Cantelli class. This follows from similar arguments as those used to establish that $ \mathcal{ F }_1 $ is Glivenko-Cantelli. 
\end{proof}

Lemma~\ref{lemScalingConv}, in conjunction with the continuity of the kernel shape function used throughout this work (see Section~\ref{secVSAKernel}) implies in the following:

\begin{cor}\label{corKConv} The data-driven kernel $ k_{Q,N} $ converges uniformly to $ k_Q $; that is, 
  \begin{displaymath}
    \lim_{N\to\infty} \lVert k_{Q,N} - k_{Q,N} \rVert_{C(\mathcal{V}\times\mathcal{V})} = 0.
  \end{displaymath} 
\end{cor}

We now proceed to bound the second term in~\eqref{eqEtaN}, $ \lVert \tilde P_{Q,N} f - \hat P_{Q,N} f \rVert_{C(\mathcal{V})} $. It follows from the definition of the kernels  $ p_{Q,N} $ and $ \hat p_{Q,N} $ via~\eqref{eqPNKernel}  and~\eqref{eqPNHatKernel}, respectively, that 
\begin{displaymath}
  \lVert \tilde P_{Q,N} f - \hat P_{Q,N} f \rVert_{C(\mathcal{V})}  \\
  \leq \lVert k_Q \rVert_{C(\mathcal{V}\times\mathcal{V})} \lVert f \rVert_{C(\mathcal{V})} 
  \left \lVert \frac{ 1 }{ l_{Q,N} \otimes r_{Q,N} } - \frac{ 1 }{ l_Q \otimes r_Q } \right \rVert_{C(\mathcal{V}\times\mathcal{V})}.
\end{displaymath}
By our assumptions on kernels stated in Section~\ref{secKernelDelays}, the functions $ l_Q $, $ r_Q $, $ l_{Q,N} $, and $ r_{Q,N} $ are bounded away from zero on $ \mathcal{V}$, uniformly with respect to $N$. Therefore, there exists a constant $ c > 0  $, independent of $ N $, such that   
\begin{align*}
  \lVert \tilde P_{Q,N} f - \hat P_{Q,N} f \rVert_{C(\mathcal{V})}  
   & \leq c \lVert k_Q \rVert_{C(\mathcal{V}\times\mathcal{V})} \lVert f \rVert_{C(\mathcal{V})} \lVert l_Q \otimes r_Q  -  l_{Q,N} \otimes r_{Q,N}  \rVert_{C(\mathcal{V}\times\mathcal{V})} \\
  & = c \lVert k_Q \rVert_{C(\mathcal{V}\times\mathcal{V})} \lVert f \rVert_{C(\mathcal{V})} 
   \lVert l_Q - l_{Q,N} \rVert_{C(\mathcal{V})} \lVert r_Q - r_{Q,N} \rVert_{C(\mathcal{V})}. 
\end{align*}
Observe now that 
\begin{align*}
  \lVert r_{Q} - r_{Q,N}  \rVert_{C(\mathcal{V})} &= \max_{\omega\in\mathcal{V}} \lvert r_{Q}( \omega ) - r_{Q,N}( \omega ) \rvert \\
  &= \max_{\omega\in\mathcal{V}} \lvert \mathbb{E} k_Q( \omega, \cdot ) - \mathbb{ E }_N k_{Q,N}( \omega, \cdot ) \rvert \\
  & \leq \max_{\omega\in\mathcal{V}} \lvert \mathbb{E} k_Q( \omega, \cdot ) - \mathbb{ E }_N k_{Q}( \omega, \cdot ) \rvert 
   + \max_{\omega\in\mathcal{V}} \lvert \mathbb{E}_N\left( k_{Q,N}( \omega, \cdot ) - k_Q( \omega, \cdot ) \right)\rvert\\
  & \leq \max_{\omega\in\mathcal{V}} \lvert \mathbb{E} k_Q( \omega, \cdot ) - \mathbb{ E }_N k_{Q}( \omega, \cdot ) \rvert 
   +\lVert k_{Q,N} - k_Q \rVert_{C(\mathcal{V}\times\mathcal{V})}.
\end{align*}
Since $ \lVert k_{Q,N} - k_Q \rVert_{C(\mathcal{V}\times\mathcal{V})} $ converges to zero by Corollary~\ref{corKConv}, it follows that  
\begin{equation}
  \label{eqRQConv}
  \lim_{N\to\infty}\lVert r_{Q} - r_{Q,N}  \rVert_{C(\mathcal{V})} = 0
\end{equation}   
if it can be shown that 
\begin{displaymath}
  \mathcal{ F }_3 = \{ k_Q( \omega, \cdot ) \mid \omega \in \mathcal{V} \}
\end{displaymath}
is a Glivenko-Cantelli class. The latter can be verified by means of similar arguments as those used to establish that $ \mathcal{ F }_1 $ is Glivenko-Cantelli. Equation~\eqref{eqRQConv}, in conjunction with the fact that $ \lVert l_Q - l_{Q,N} \rVert_{C(\mathcal{V})} $ is bounded, is sufficient to deduce that $ \lim_{N\to\infty}\lVert \tilde P_{Q,N} f - \hat P_{Q,N} f \rVert_{C(\mathcal{V})} = 0$. 

We now turn to the third term in~\eqref{eqEtaN}, $ \lVert \hat P_{Q,N} f - P_{Q,N} f \rVert_{C(\mathcal{V})} $. 
We have
\begin{displaymath}
  \lVert \hat P_{Q,N} f - P_{Q,N} f \rVert_{C(\mathcal{V}\times\mathcal{V})} \leq \lVert \hat p_{Q,N} - p_{Q,N} \rVert_{C(\mathcal{V}\times\mathcal{V})} \lVert f \rVert_{C(\mathcal{V})},
\end{displaymath}
and it follows from the definitions of $ \hat p_{Q,N} $ and $ p_{Q,N} $, in conjunction with the fact that the normalization functions $ l_{Q,N} $ and $ r_{Q,N} $ are both uniformly bounded away from zero, that there exists a constant $ c $ such that
\begin{displaymath}
   \lVert \hat P_{Q,N} f - P_{Q,N} f \rVert_{C(\mathcal{V}\times\mathcal{V})} \leq c \lVert k_Q - k_{Q,N} \rVert_{C(\mathcal{V}\times\mathcal{V})} \lVert f \rVert_{C(\mathcal{V})}.
\end{displaymath}
Thus, the convergence of $ \lVert \hat P_{Q,N} f - P_{Q,N} f \rVert_{\mathcal{V}} $ to zero follows from Corollary~\ref{corKConv}. 

In summary, we have shown that $ \lVert P_{N} f - \tilde  P_{Q,N} f \rVert_{C(\mathcal{V})} $, $ \lVert \tilde P_{Q,N} f - \hat P_{Q,N} f \rVert_{C(\mathcal{V})} $,  and $ \lVert \hat P_{Q,N} f - P_{Q,N} f \rVert_{C(\mathcal{V})} $ all converge to zero, which is sufficient to conclude that $ \lim_{N\to\infty} \eta_N = 0 $, and that that $ P_{Q,N} f $ converges to $ P_Q f $. 

According to Definition~\ref{defCC}, it remains to show that for any bounded sequence $ f_N \in C( \mathcal{V} ) $, the sequence $ g_N = ( P_{Q,N} - P_Q ) f_N $ has a limit point. This can be proved by  an Arzel\`a-Ascoli argument as in the proof of Claim (a) in conjunction with Glivenko-Cantelli arguments as in the proof of pointwise convergence above. We refer the reader to Proposition~13 in \cite{VonLuxburgEtAl08} for more details. This completes our proof of Claim (b).


\end{document}